\documentclass[10pt]{article}
\usepackage{pifont}
\usepackage{amsmath}
\usepackage{amscd}
\usepackage{amsfonts}
\usepackage{latexsym}
\usepackage{amssymb}
\usepackage{stmaryrd}
\usepackage{overpic}
\usepackage{graphicx}
\usepackage{bm}
\usepackage{subcaption}
\usepackage{empheq}
\usepackage{multirow}
\usepackage{enumitem}
\usepackage{lineno}
\usepackage{rotating}
\usepackage[unicode,colorlinks,linkcolor=blue,hyperindex]{hyperref}

\usepackage{tikz}
\usetikzlibrary{er}

\newtheorem{theorem}{Theorem}[section]

\newtheorem{lemma}[theorem]{Lemma}

\newtheorem{remark}[theorem]{Remark}
\newenvironment{proof}[1][Proof]{\textbf{#1.} }
{\ \rule{0.75em}{0.75em}\smallskip}

\numberwithin{equation}{section}
\numberwithin{table}{section}
\numberwithin{figure}{section}

\newcommand{\vect}[1]{\boldsymbol{#1}} 

\begin{document}

\title{An Extended Galerkin Analysis for Elliptic Problems
\thanks{The work of Jinchao Xu was supported in part by US Department
of Energy Grant DE-SC0014400 and NSF grant DMS-1522615. The work of
Shuonan Wu was supported in part by the startup grant from Peking
University.}}

\author{Qingguo Hong\footnote{huq11@psu.edu@psu.edu, Department of
Mathematics, Pennsylvania State University, University Park, PA,
16802, USA} \quad 
Shuonan Wu\footnote{snwu@math.pku.edu.cn, School of Mathematical
Sciences, Peking University, Beijing, 100871, China}
\quad
Jinchao Xu\footnote{xu@math.psu.edu, Department of Mathematics,
Pennsylvania State University, University Park, PA, 16802, USA}}
\date{}
\maketitle 

\begin{abstract}
A general analysis framework is presented
in this paper for many different types of finite element methods
(including various discontinuous Galerkin methods).  For second order
elliptic equation, this framework employs $4$ different discretization
variables, $u_h, \bm{p}_h, \check u_h$ and $\check p_h$, where $u_h$
and $\bm{p}_h$ are for approximation of $u$ and $\bm{p}=-\alpha \nabla
u$ inside each element,  and $ \check u_h$ and $\check p_h$ are for
approximation of residual of $u$ and $\bm{p} \cdot \bm{n}$ on the
boundary of each element.  The resulting 4-field discretization is
proved to satisfy inf-sup conditions that are uniform with respect to
all discretization and penalization parameters.  As a result, most
existing finite element and discontinuous Galerkin methods can be
analyzed using this general framework by making appropriate choices of
discretization spaces and penalization parameters.
\end{abstract}


\section{Introduction} \label{sec:intro} 
In this paper, we propose an {\it extended Galerkin} analysis
framework for most of the existing finite element methods (FEMs).  We
will illustrate the main idea by using the following elliptic boundary
value problem
\begin{equation} \label{pois}
\left\{
\begin{aligned}
-{\rm div} (\alpha\nabla u)&= f ~~\quad {\rm in} \ \Omega, \\
 u&=g_D \quad {\rm on} \ \Gamma_D,\\
-(\alpha \nabla u) \cdot \bm{n} &=g_N \quad {\rm on} \ \Gamma_N,
\end{aligned}
\right.
\end{equation}
where $\Omega\subset \mathbb R^d$ ($d\ge 1$) is a bounded domain and
its boundary, $\partial\Omega$, is split into Dirichlet and Neumann
parts, namely $\partial\Omega = \Gamma_D \cup \Gamma_N$. For
simplicity, we assume that the $(d-1)$-dimensional measure of
$\Gamma_D$ is nonzero. Here $\bm n$ is the outward unit normal
direction of $\Gamma_N$, and $\alpha: \mathbb{R}^d \rightarrow
\mathbb{R}^d$ is a bounded and symmetric positive definite matrix,
with its inverse denoted by $c =\alpha^{-1}$.  Setting $\bm{p}
=-\alpha\nabla u$, the above problem can be written as 
\begin{equation}\label{H1}
\left\{
\begin{aligned}
c \bm{p} +\nabla u &= 0 & {\rm in}\ \Omega, \\ 
-{\rm div}  \bm{p} &= -f & {\rm in}\ \Omega, 
\end{aligned}
\right.
\end{equation}   
with the boundary condition $u=g_D~{\rm on}~\Gamma_D$ and $\bm
p\cdot\bm n=g_N~{\rm on}~\Gamma_N$.

There are two major variational formulations for \eqref{pois}.  The
first is to find $u\in H^1_D(\Omega):= \{v\in H^1(\Omega):
v|_{\Gamma_D} = g_D\}$ such that for any 
$v\in H_{D0}^1(\Omega) :=\{ v\in H^1(\Omega):~ v|_{\Gamma_D} = 0\}$, 
\begin{equation} \label{H1-vari}
\int_\Omega (\alpha\nabla u) \cdot \nabla v \,\mathrm{d}x = 
\int_\Omega fv \,\mathrm{d}x -
\int_{\Gamma_N} g_N v\,\mathrm{d}s.
\end{equation}
The second one is to find $\bm{p}\in H_N({\rm div};\Omega):=
  \{\bm{q}\in H({\rm div}):~\bm{q}\cdot \bm{n} = g_N\}$, $u\in
L^2(\Omega)$ such that for any $ \bm{q} \in H_{N0}({\rm div};\Omega):= 
\{\bm{q}\in H({\rm div}):~\bm{q}\cdot \bm{n} = 0\}$ and $v\in
L^2(\Omega)$,
\begin{equation} \label{mixed-vari}
\left\{
\begin{aligned}
\int_{\Omega} \bm{p}\cdot \bm{q}\,\mathrm{d}x - \int_\Omega u 
\,{\rm div}\bm{q}\,\mathrm{d}x &= -\int_{\Gamma_D} g_D \bm{q}\cdot
\bm{n} \,\mathrm{d}s, \\
-\int_{\Omega} v \, {\rm div}\bm{p} \,\mathrm{d}x &= -\int_\Omega f v
\,\mathrm{d}x. 
\end{aligned}
\right.
\end{equation}
In correspondence to the two variational formulations, two different
conforming finite element methods have been developed. The first one,
which approximates $u \in H_D^1(\Omega)$, can be traced back to the
1940s \cite{hrennikoff1941solution} and the Courant element
\cite{courant1943variational}.  After a decade, many works, such as
\cite{jones1964generalization, Feng1965finite,
fraeijs1965displacement, Yamamoto1966, zlamal1968finite,
pin1969variational, ciarlet1971multipoint, nicolaides1972class},
proposed more conforming elements and presented serious mathematical
proofs concerning error analysis and, hence, established the basic
theory of FEMs. These {\it primal FEMs} contain one unknown, namely
$u$, to solve.  The second one, which approximates $\bm{p} \in
H_N({\rm div};\Omega)$ and $u \in L^2(\Omega)$ based on a mixed
variational principal, is called the {\it mixed FEMs}
\cite{brezzi1974existence,raviart1977mixed,
nedelec1980mixed,arnold1985mixed,brezzi1991mixed,boffi2013mixed}.
These mixed methods solve two variables, namely flux variable $\bm p$
and $u$, and the condition for the well-posedness of mixed
formulations is known as inf-sup or the
Ladyzhenskaya-Babu\v{s}ka-Breezi (LBB) condition
\cite{brezzi1974existence}.

Contrary to the continuous Galerkin methods, the discontinuous
Galerkin (DG) methods, which can be traced back to the late 1960s
\cite{lions1968penalty, aubin1970approximation}, aim to relax the
conforming constraint on $u$ or $\bm{p}\cdot \bm{n}$.  To maintain
consistency of the DG discretization, additional finite element spaces
need to be introduced on the element boundaries.  In essence, the
numerical fluxes on the element boundaries were introduced explicitly
and therefore eliminated.  In most existing DG methods, only one such
boundary space is introduced as, for example, Lagrangian multiplier
space, either for $u$ as the primal DG methods
\cite{douglas1976interior,cockburn1998local,brezzi2000discontinuous}
or for $\bm{p}\cdot \bm{n}$ as the mixed DG methods
\cite{hong2017unified}.  Primal DG methods have been applied to purely
elliptic problems; examples include the interior penalty methods
studied in \cite{babuvska1973nonconforming, wheeler1978elliptic,
arnold1982interior} and the local DG method for elliptic problem in
\cite{cockburn1998local}.  Primal DG methods for diffusion and
elliptic problems were considered in \cite{brezzi1999discontinuous}.
A review of the development of DG methods up to 1999 can be found in
\cite{cockburn2000development}.  

Given a triangulation of $\Omega$, let $u_h \in V_h$ and $\bm{p}_h \in
\bm{Q}_h$ be discontinuous piecewise polynomial approximations of $u$
and $\bm p$, respectively.  In \cite{arnold2002unified}, Arnold,
Brezzi, Cockburn, and Marini unified the analysis of DG methods for
elliptic problems \eqref{pois} with $c=1$ and $\Gamma_D =
\partial\Omega$, which hinges on the unified formulation \cite[Equ.
(3.11)]{arnold2002unified} (Here, we change the notation
$\widehat{u}_h \mapsto \bar{u}_h$, $\sigma_h \mapsto - \bm{p}_h$ and
$\widehat{\sigma}_h \mapsto -\bar{\bm{p}}_h$, see also
\eqref{eq:DG-notation} for the DG notation):
\begin{equation} \label{eq:ABCM} 
\begin{aligned}
(\nabla_h u_h, \nabla_h v_h) 
& + \langle \llbracket \bar u_h - u_h \rrbracket , \{\nabla_h
v_h\}\rangle_{\mathcal{E}_h}
+ \langle \{\bar{\bm p}_h\}, \llbracket v_h \rrbracket
\rangle_{\mathcal{E}_h} \\
& + \langle \{\bar u_h - u_h\}, \llbracket\nabla_h v_h \rrbracket
\rangle_{\mathcal{E}_h^i} + \langle \llbracket\bar{\bm p}_h \rrbracket
, \{v_h\} \rangle_{\mathcal{E}_h^i} =(f, v_h) \qquad \forall v_h \in
V_h,
\end{aligned}
\end{equation}
where the numerical traces $\bar u_h$ and $\bar{\bm p}_h$ (i.e.,
$-\widehat{\sigma}_h$ in \cite{arnold2002unified}) are explicitly
given in \cite[Table 3.1]{arnold2002unified}. 

As a key step in our extended Galerkin analysis, we introduce two additional
residual corrections to the numerical traces $\bar{u}_h$ and $\bar{\bm
p}_h$ in \eqref{boundary}, which gain the flexibility of boundary
finite element spaces for both $u$ and $\bm{p}\cdot \bm{n}$.  More
specifically, in addition to the $\bar{u}_h$ and $\bar{\bm p}_h$ given
explicitly, our extended Galerkin analysis is presented in terms of four
discretization variables, namely 
$$
\bm{p}_h, \quad \check{p}_h, \quad u_h, \quad \check{u}_h.
$$
 The variables $\check{u}_h$ and $\check{p}_h$ are
introduced for the following approximation on element boundary 
$$
u \approx \bar{u}_h + \check u_h, \qquad
\bm{p}\cdot \bm{n}_e \approx \bar{\bm{p}}_h\cdot \bm{n}_e + \check p_h
, \quad e=K^+\cap K^-,
$$
which gives the following formulation by adopting the DG notation
\eqref{eq:DG-notation}, 
\begin{equation} \label{eq:arnold-mixed}
\left\{
\begin{aligned}
(c\bm{p}_h, \bm{q}_h) - (u_h, {\rm div}_h \bm{q}_h) + \langle
\bar{u}_h + \check{u}_h, [\bm{q}_h] \rangle_{\mathcal E_h} &= -\langle
g_D, \bm{q}_h \cdot \bm{n} \rangle_{\Gamma_D}
~~\quad \quad \qquad \forall \bm{q}_h \in \bm{Q}_h,\\
(\bm{p}_h, \nabla_h v_h) - \langle \bar{\bm{p}}_h\cdot \bm{n}_e +
\check{p}_h, [v_h]_e \rangle_{\mathcal E_h} &= -(f, v_h) + \langle
g_N, v_h \rangle_{\Gamma_N} 
\qquad \forall v_h \in V_h.
\end{aligned}
\right.
\end{equation}
As a direct consequence, the formulation \eqref{eq:arnold-mixed} is
equivalent to the formulation \cite[Equ.
(3.4)-(3.5)]{arnold2002unified} if we simply choose $\check{u}_h =
\check{p}_h = 0$ and $c=1$, which leads to \eqref{eq:ABCM} by eliminating
$\bm{p}_h$ (i.e., $-\sigma_h$ in \cite{arnold2002unified}).  As in
most DG methods, the Nitsche's trick (see
\eqref{up0} below) for $\check u_h$ and $\check p_h$ will be used.  In
this paper, we develop a concise {\it formulation} (see
\eqref{eq:XG} below) in terms of four variables $\bm{p}_h,
\check{p}_h, u_h, \check{u}_h$, which contain all the possible
variables in most of the existing FEMs.  Therefore, it has the
flexibility to unify the analysis of most of the existing FEMs: 

\begin{enumerate}
\item Under proper choices of the discrete spaces,  formulation \eqref{eq:XG}
recovers the analysis of $H^1$ conforming finite element if we
eliminate all the discretization variables except $u_h$. By
eliminating $\check{p}_h$, formulation \eqref{eq:XG} recovers some special cases
of the hybrid methods \cite{cockburn2004characterization} in which
$\check{u}_h$ serves as a Lagrange multiplier to force the continuity
of $\bm p\cdot \bm n$ across the element boundary. If we further
eliminate the Lagrange multiplier, the resulting system needs to solve
two variables $\bm{p}_h$ and $u_h$, which recovers the $H({\rm div})$
conforming mixed finite element method. 

\item The relationship between the formulation \eqref{eq:XG} and DG methods is
twofold.  First, by simply taking the trivial spaces for $\check u_h$
and $\check p_h$,  formulation \eqref{eq:XG} recovers most of DG methods shown in
\cite{arnold2002unified}. Second, if we confine to a special choice
$\bar{u}_h = \{u_h\}$ and $\bar{\bm{p}}_h = \{\bm{p}_h\}$, by virtue
of the characterization of hybridization and DG method
\cite{cockburn2004characterization}, formulation \eqref{eq:XG} can be related to
most of DG methods if we eliminate both $\check{p}_h$ and
$\check{u}_h$ (see Section \ref{subsec:DG}).

\item In Section \ref{subsec:HDG}, formulation \eqref{eq:XG} can be compared with
most hybridized discontinuous Galerkin (HDG) methods if we eliminate
$\check{p}_h$.  In 2009, a unified formulation of the hybridization of
discontinuous Galerkin, mixed, and continuous Galerkin methods for
second order elliptic problems was presented in
\cite{cockburn2009unified}.  The resulting system needs to solve three
variables, one approximating $u$, one approximating $p$, and the third
one approximating the trace of $u$ on the element boundary. A
projection-based error analysis of HDG methods was presented in
\cite{cockburn2010projection}, in which a projection operator was
tailored to obtain the $L^2$ error estimates for both potential and
flux. 
More references to the recent developments of HDG methods can be found
in \cite{cockburn2016static}. 

\item In Section \ref{subsec:WG}, formulation \eqref{eq:XG} can be compared with
most weak Galerkin (WG) methods if we eliminate $\check{u}_h$.  With
the introduction of weak gradient and weak divergence, a WG method for
a second-order elliptic equation formulated as a system of two
first-order linear equations was proposed and analyzed in
\cite{wang2013weak, wang2014weak}.  In fact, the weak Galerkin methods
in \cite{wang2014weak} also solve three variables, one approximating
$u$, one approximating $\bm p$, and the third one approximating the
flux $\bm p\cdot \bm n$ on the element boundary.  A summary of the
idea and applications of WG methods for various problems can be found
in \cite{wang2015weak}.  
\end{enumerate}


In addition, we study two types of uniform inf-sup conditions for the
proposed formulation in Section \ref{sec:well-posedness}, by which the
well-posedness of the formulation \eqref{eq:XG} follows naturally. With these
uniform inf-sup conditions, we obtain some limiting of formulation \eqref{eq:XG}
in Section \ref{sec:limiting}:
\begin{enumerate}
\item If the parameters in the Nitsche's trick are set to be $\tau =
(\rho h_e)^{-1}$, $\eta\cong \tau^{-1}$,  formulation \eqref{eq:XG} is shown to
converge to $H^1$ conforming method as $\rho \to 0$ under certain
conditions pertaining to the discrete spaces.
\item If the parameters in the Nitsche's trick are set to be $\eta =
(\rho h_e)^{-1}$, $\tau\cong \eta^{-1}$,  formulation \eqref{eq:XG} is shown to
converge to $H({\rm div})$ conforming method as $\rho \to 0$ under
certain conditions pertaining to the discrete spaces.
\end{enumerate}

Throughout this paper, we shall use letter $C$, which is independent
of mesh-size and stabilization parameters, to denote a generic
positive constant which may stand for different values at different
occurrences.  The notations $x \lesssim y$ and $x \gtrsim y$ mean $x
\leq Cy$  and $x \geq Cy$, respectively.

\section{Preliminaries} \label{sec:preliminaries}
Given $\Omega \subset \mathbb{R}^d$,  for any 
$D\subseteq \Omega$, and any positive integer
$m$, let $H^m(D)$ be the Sobolev space with the corresponding usual norm
and semi-norm, denoted by $\|\cdot\|_{m,D}$ and
$|\cdot|_{m,D}$, respectively.  The
$L^2$-inner product on $D$ and $\partial D$ are denoted by $(\cdot,
\cdot)_{D}$ and $\langle\cdot, \cdot\rangle_{\partial D}$,
respectively.  $\|\cdot\|_{0,D}$ and $\|\cdot\|_{0,\partial D}$ are
the norms of Lebesgue spaces $L^2(D)$ and $L^2(\partial D)$,
respectively.
We abbreviate $\|\cdot\|_{m,D}$ and
$|\cdot|_{m,D}$ by $\|\cdot\|_{m}$
and $|\cdot|_{m}$, respectively, when $D=\Omega$, 
and $\|\cdot\|_0=\|\cdot\|_{0,\Omega}$.

\subsection{DG notation}
We denote by $\{\mathcal{T}_h\}_h$ a family of
shape-regular triangulations of $\overline{\Omega}$. Let $h_K ={\rm
diam}(K)$ and $h = \max\{h_K: K\in \mathcal{T}_h\}$.  For any $K\in
\mathcal{T}_h$, denote $\bm{n}_K$ as the outward unit normal of $K$.
Denote by ${\cal E}_h$ the union of the boundaries of the elements $K$
of $\mathcal{T}_h$. 

Let ${\cal E}_h^i={\cal E}_h \setminus \partial\Omega$ be the set of
interior edges and ${\cal E}_h^\partial={\cal E}_h\setminus{\cal
E}_h^i$ be the set of boundary edges. Further, for any $e\in \mathcal
E_h$, let $h_e ={\rm diam}(e)$. For $e\in {\cal E}_h^i$, we select a
fixed normal unit direction, denoted by $\bm{n}_e$. For $e\in {\cal
E}_h^\partial$, we specify the unit outward normal of $\Omega$ as $\bm{n}_e$.  Let
$e$ be the common edge of two elements $K^+$ and $K^-$, and let
$\bm{n}^i$ = $\bm{n}|_{\partial K^i}$ be the unit outward normal
vector on $\partial K^i$ with $i = +,-$.  For any scalar-valued
function $v$ and vector-valued function $\bm{q}$, let $v^{\pm}$ =
$v|_{\partial K^{\pm}}$, $\bm{q}^{\pm}$ = $\bm{q}|_{\partial
K^{\pm}}$.  Then, we define averages $\{\cdot\}$, $\{\!\!\{ \cdot
\}\!\!\}$, $\{\cdot \}_e$ and jumps $\llbracket \cdot \rrbracket$,
$[\cdot]_e$, $[\cdot]$ as follows:
\begin{equation} \label{eq:DG-notation}
\begin{aligned}
&\{v\} = \frac{1}{2}(v^+ + v^-),\qquad \{\!\!\{\bm{q}\}\!\!\} =
\frac{1}{2}(\bm{q}^+ + \bm{q}^-), \qquad \{\bm
q\}_e = \frac{1}{2}(\bm{q}^+ + \bm{q}^-)\cdot \bm{n}_e ~~~\quad &{\rm
on}\ e\in {\cal E}_h^i,\\
&\llbracket v \rrbracket  = v^+\bm{n}^+ + v^-\bm{n}^- ,\qquad [v]_e =
\llbracket v \rrbracket\cdot \bm{n}_e,
\qquad[\bm{q}] = \bm{q}^+\cdot \bm{n}^+ + \bm{q}^-\cdot\bm{n}^-
\quad &{\rm on}\ e\in {\cal E}_h^i,\\
&\llbracket v \rrbracket = v \bm{n}, \qquad[v]_e = v, 
\qquad  \{v\} = v,  \qquad \{\!\!\{\bm{q}\}\!\!\} = \bm{q},\qquad
\{\bm q\}_e=\bm q\cdot \bm n, \qquad [\bm q]=0 \qquad &{\rm on}\ e \in
\Gamma_D,\\
&\llbracket v \rrbracket = \bm 0, \qquad[v]_e = 0, 
\qquad  \{v\} = v,  
\qquad \{\!\!\{\bm{q}\}\!\!\} =\bm{q},\qquad  \{\bm q\}_e = \bm q\cdot
\bm n, \qquad [\bm q]=\bm q\cdot \bm n \qquad &{\rm on}\ e \in
\Gamma_N.
\end{aligned}
\end{equation}
The notation follows the
rules: (i) $\{\!\!\{ \cdot \}\!\!\}$ and $\llbracket \cdot
\rrbracket$ are vector-valued operators; (ii) $\{\cdot\}$, $[\cdot]$,
$\{\cdot\}_e$ and $[\cdot]_e$ are scalar-valued operators; (iii)
$\{\cdot\}_e$ and $[\cdot]_e$ are orientation-dependent operators.
Clearly, $\{\!\!\{\bm{q}\}\!\!\} \cdot \llbracket v \rrbracket =
\{\bm{q}\}_e [v]_e$.

For simplicity of exposition, we use the following convention:
\begin{equation} \label{equ:inner-product}
(\cdot,\cdot):=\sum_{K\in \mathcal
T_h}(\cdot,\cdot)_{K}, \qquad \langle\cdot,\cdot\rangle:=\sum_{e\in \mathcal E_h}\langle\cdot,\cdot\rangle_{e},
\qquad \langle\cdot,\cdot\rangle_{\partial\mathcal T_h }:=\sum_{K\in
\mathcal T_h}\langle\cdot,\cdot\rangle_{\partial K}.
\end{equation}
We now give more details about the last
inner product. For any scalar-valued function $v$ and vector-valued
function $\bm q$, we denote
$$
\langle v, \bm  q \cdot \bm n \rangle_{\partial
\mathcal T_h}
:=\sum_{K\in \mathcal T_h}\langle v,
\bm q\cdot \bm n_K\rangle_{\partial K}.
$$
Here, we specify the outward unit normal direction $\bm n$
corresponding to the element $K$, namely $\bm n_K$.  In
addition, let $\nabla_h$ and ${\rm div}_h$ be defined as
$$
\nabla_hv|_K:=\nabla v|_K, \quad
{\rm div}_h\bm{q}|_K:={\rm div} \bm{q}|_K \qquad \forall K \in
\mathcal{T}_h.
$$

\begin{lemma}
With the averages and jumps defined in
\eqref{eq:DG-notation}, we have the following identities
\cite{arnold2002unified}:
\begin{align}
(v,{\rm div}_h \bm q)+(\nabla_h v, \bm q) &= 
\langle v, \bm q\cdot \bm n\rangle_{\partial \mathcal T_h}
=\langle \{\!\!\{\bm q\}\!\!\}, \llbracket v
\rrbracket\rangle+ \langle [\bm
q],\{v\}\rangle 
= \langle \{\bm q\}_e, [v]_e \rangle+ \langle [\bm
q],\{v\}\rangle, \label{eq:dg-identity} \\
\langle u_h, v_h\rangle_{\partial \mathcal T_h} &= 
2\langle \{u_h\},\{v_h\}\rangle+ \frac{1}{2}\langle
\llbracket u_h \rrbracket,\llbracket v_h \rrbracket\rangle= 
2\langle \{u_h\},\{v_h\}\rangle+ \frac{1}{2}\langle
[u_h]_e,[v_h]_e\rangle.  \label{eq:dg-identity_2}
\end{align}
\end{lemma}
\begin{proof}
On each $e=\partial K^+\cap \partial K^-$, the following identity can
be verified by a direct calculation:
\begin{equation} \label{DG-identity}
\bm{q}^+\cdot \bm{n}^+v^+ + \bm{q}^-\cdot\bm{n}^- v^- =
\{\!\!\{\bm{q}\}\!\!\} \cdot \llbracket v \rrbracket+[\bm{q}]\{v\}.
\end{equation}
Consequently, by the averages and jumps defined on $\Gamma_D$ and
$\Gamma_N$ in \eqref{eq:DG-notation}, we have 
\begin{equation}\label{DG-identity-2}
\langle v, \bm q\cdot \bm n\rangle_{\partial \mathcal T_h} =\langle
\llbracket v \rrbracket, \{\!\!\{\bm{q}\}\!\!\} \rangle+ \langle
\{v\}, [\bm q] \rangle.
\end{equation}
By integrating by parts and  \eqref{DG-identity-2}, we have identity
\eqref{eq:dg-identity}.  Identity \eqref{eq:dg-identity_2} can be
obtained by a direct calculation. 
\end{proof}
\paragraph{DG finite element spaces.}
Before discussing various Galerkin methods, we need to introduce the
finite element spaces associated with the triangulation
$\mathcal{T}_h$.  First, $V_h$ and $\bm{Q}_h$ are the piecewise scalar
and vector-valued discrete spaces on the triangulation
$\mathcal{T}_h$, respectively and for $k\ge 0$, we define the spaces
as follows
\begin{equation}\label{Spaces}
\begin{aligned}
V^{k}_h&:=\{v_h\in L^2(\Omega): v_h|_{K}\in \mathcal{P}_k(K), \forall
K\in \mathcal T_h \},\\
\bm Q^k_h&:= \{\bm p_h\in \bm L^2(\Omega): \bm p_h|_K\in \bm
{\mathcal{P}}_k(K), \forall K\in \mathcal T_h \},\\
\bm Q^{k,RT}_h &:= \{\bm p_h\in \bm L^2(\Omega): \bm p_h|_K\in \bm
{\mathcal{P}}_k(K)+\bm x \mathcal{P}_k(K), \forall K\in \mathcal T_h \},
\end{aligned}
\end{equation}
where $\mathcal{P}_k(K)$ is the space of polynomial functions of
degree at most $k$ on $K$, and $\bm{\mathcal P}_k(K) := [\mathcal P_k(K)]^d$.   

Second, $\check{V}_h$ and $\check{Q}_h$ are the piecewise
scalar-valued discrete spaces on ${\cal E}_h$, respectively and for
$k\ge 0$, we define the spaces as follows
\begin{equation}\label{Edge:spaces}
\begin{aligned}
{\check Q}^k_h&:=\{{\check p}_h\in L^2(\mathcal E_h): {\check
p}_h|_e\in \mathcal{P}_k(e), \forall e\in \mathcal{E}^i_h,{\check
p}_h|_{\Gamma_N}=0\}, \\
{\check V}^k_h&:=\{{\check v}_h \in L^2(\mathcal E_h): {\check
v}_h|_e\in \mathcal{P}_k(e), \forall e\in \mathcal{E}^i_h, {\check
  v}_h|_{\Gamma_D}=0\},
\end{aligned}
\end{equation}
where $\mathcal{P}_k(e)$ is the space of polynomial functions of
degree at most $k$ on $e$. Further, let $\check Q(e), \check V(e)$
denote some local spaces on $e$ which will be specified at their
occurrences.

\section{A Unified Four Field Formulation}\label{SXG:formulation}
We start with equation \eqref{H1}, namely
\begin{equation}\label{H11}
\left\{
\begin{aligned}
c \bm{p} +\nabla u &= 0 \,\qquad  {\rm in}\ \Omega, \\ 
-{\rm div} \bm{p} &=-f  ~\quad {\rm in}\ \Omega. \\
\end{aligned}
\right.
\end{equation}
Multiplying the first and second equations by $\bm q_h\in \bm Q_h$ and
$v_h\in V_h$, and summing on all $K\in \mathcal T_h$, we get 
\begin{equation*} 
\left\{
\begin{aligned}
(c\bm p, \bm q_h) +(\nabla u, \bm q_h)&= 0 \qquad \qquad\quad \forall
\bm q_h\in \bm Q_h,\\
-({\rm div} \bm p, v_h) &=-(f, v_h)~\qquad \forall
v_h\in V_h.
\end{aligned}
\right.
\end{equation*}
Using the identity \eqref{eq:dg-identity}, we have 
\begin{equation*} 
\left\{
\begin{aligned}
(c\bm p, \bm q_h) -(u, {\rm div}_h\bm q_h) + \langle u, [\bm
q_h]\rangle + \langle [u]_e, \{\bm
q_h\}_e\rangle &= 0 \qquad \qquad\quad \forall \bm
q_h\in \bm Q_h,\\
(\bm p, \nabla_h v_h) -
\langle \bm{p}\cdot \bm{n}_e, [v_h]_e \rangle
-\langle [\bm p], \{v_h\}\rangle
&=-(f, v_h)~\qquad \forall
v_h\in V_h.
\end{aligned}
\right.
\end{equation*}
Noting that $u\in H^1(\Omega), \bm p\in \bm{H}({\rm div},\Omega)$,
$u=g_D$ on $\Gamma_D$ and $\bm p\cdot \bm n=g_N$ on $\Gamma_N$, we
obtain 
\begin{equation} 
\left\{
\begin{aligned}
\displaystyle (c\boldsymbol p, \boldsymbol q_h)
-(u, {\rm div}_h\boldsymbol q_h)+
\langle u, [\bm q_h]\rangle
&= - \langle g_D, \bm q_h\cdot \bm n\rangle_{\Gamma_D}  
\qquad\qquad 
\forall \boldsymbol q_h\in \boldsymbol Q_h,\\
\displaystyle (\boldsymbol p, \nabla_h v_h)-
\langle \bm{p}\cdot \bm{n}_e, [v_h]_e\rangle&=-(f, v_h)
+ \langle g_N, v_h\rangle_{\Gamma_N}~~\quad \forall
v_h\in V_h.
\end{aligned}
\right.
\end{equation}

\paragraph{The unified formulation.}
It is natural to approximate $u, \bm p$ on the interior of the elements of $\mathcal T_h$ by 
\begin{equation} \label{eq:1}
\quad u\approx u_h,\quad  \boldsymbol p\approx \boldsymbol p_h, 
\end{equation}
for $u_h\in V_h$ and $\bm p_h\in \bm Q_h$. 
Our key observation is that most DG methods can be obtained by
approximating $u$ and $\bm{p}\cdot \bm{n}_e$ on $\mathcal E_h$ by
\begin{equation} \label{boundary}
u \approx\bar u_h(u_h)+\check u_h, \quad
\bm{p}\cdot \bm{n}_e\approx \bar p_h(u_h, \bm p_h)+\check p_h,
\end{equation}
where $\bar u_h(u_h)$, $\bar p_h(u_h, \bm p_h)$ are given in terms of
$u_h, \bm p_h$ as shown in \cite[Table 3.1]{arnold2002unified} (by changing the notation $\widehat{\sigma}_h \cdot \bm{n}_e \mapsto -\bar{p}_h$) 
and $\check u_h\in \check V_h$, $\check p_h\in \check Q_h$ are
some {\it residual corrections} to $\bar u_h(u_h), \bar p_h(u_h, \bm{p}_h)$,
respectively. 
As a result, we obtain
\begin{equation} \label{eq:symm-DG}
\left\{
\begin{aligned}
(c\boldsymbol p_h, \boldsymbol q_h) -(u_h, {\rm div}_h\boldsymbol
q_h)+ \langle \bar u_h(u_h)+\check u_h, [\bm q_h]\rangle &=
-\langle g_D, \bm q_h\cdot
\bm n\rangle_{\Gamma_D} 
\qquad\qquad
\forall \boldsymbol q_h\in \boldsymbol Q_h,\\
(\boldsymbol p_h, \nabla_h v_h)- \langle \bar p_h(u_h,\bm p_h)+\check
p_h, [v_h]_e\rangle &= -(f, v_h)+\langle g_N,
v_h\rangle_{\Gamma_N}~~\quad \forall v_h\in V_h.
\end{aligned}
\right.
\end{equation}
Besides \eqref{eq:symm-DG}, two additional equations are required to
determine $\check p_h$ and $\check u_h$.  On the interior edges, we
adopt 
$$ 
\check p_h\approx \tau [u_h]_e, \quad \check u_h\approx \eta [\bm p_h].
$$
More specifically, 
\begin{subequations}
\label{up0}
\begin{align}
\langle [u_h]_e -\tau^{-1} \check
p_h,\check{q}_h\rangle_{\mathcal{E}^i_h} &= 0  \qquad \forall
\check{q}_h \in \check{Q}_h,\\
\langle [\bm p_h]-\eta^{-1} \check
u_h,\check{v}_h\rangle_{\mathcal{E}^i_h} &= 0 \qquad \forall
\check{v}_h \in \check{V}_h.  
\end{align}
\end{subequations}
On the boundary edges, we naturally adopt 
$$
\check p_h\approx 
\left\{
\begin{array}{ll}
\tau(u_h-g_D) &  \hbox{on}~~\Gamma_D, \\
0 & \hbox{on}~~\Gamma_N,
\end{array}
\right.
\qquad \check u_h\approx 
\left\{
\begin{array}{ll}
0 & \hbox{on}~~\Gamma_D, \\
\eta(\bm p_h\cdot \bm n-g_N)& \hbox{on}~~\Gamma_N,
\end{array}
\right.
$$
namely 
\begin{subequations} \label{upB}
\begin{align}
-\langle u_h- \tau^{-1} \check
p_h,\check{q}_h\rangle_{\mathcal E_h^\partial} &= -\langle
g_D,\check{q}_h\rangle_{\Gamma_D}  \qquad \forall \check{q}_h \in
\check{Q}_h, \\
\langle \bm p_h\cdot \bm n- \eta^{-1} \check
u_h,\check{v}_h\rangle_{\mathcal E_h^\partial} &= \langle
g_N,\check{v}_h\rangle_{\Gamma_N} ~~\qquad \forall \check{v}_h \in
\check{V}_h.
\end{align}
\end{subequations}
Collectively, we obtain a concise formulation of \eqref{up0}-\eqref{upB} as follows
\begin{subequations} \label{eq:boundary}
\begin{align}
-\langle [u_h]_e-\tau^{-1} \check
p_h,\check{q}_h\rangle &= -\langle
g_D,\check{q}_h\rangle_{\Gamma_D}  \qquad \forall \check{q}_h \in
\check{Q}_h, \label{eq:XG3}\\
\langle [\bm p_h]- \eta^{-1} \check
u_h,\check{v}_h\rangle &= \langle
g_N,\check{v}_h\rangle_{\Gamma_N} ~~\qquad \forall \check{v}_h \in
\check{V}_h.\label{eq:XG4}
\end{align}
\end{subequations}
The combination of \eqref{eq:symm-DG} and \eqref{eq:boundary} obtains formulation: Find $(\bm p_h, \check
p_h, u_h, \check
u_h)\in \bm Q_h\times\check Q_h\times V_h\times \check V_h$ such
that for any $(\bm q_h, \check q_h,v_h,  \check v_h)\in \bm
Q_h\times\check Q_h\times V_h\times \check V_h$
\begin{equation} \label{eq:symm-ABCMDG}
\left\{
\begin{aligned}
(c\boldsymbol p_h, \boldsymbol q_h) -(u_h, {\rm div}_h\boldsymbol
q_h)+ \langle \bar u_h(u_h)+\check u_h, [\bm q_h]\rangle &=
-\langle g_D, \bm q_h\cdot
\bm n\rangle_{\Gamma_D} 
\qquad\qquad
\forall \boldsymbol q_h\in \boldsymbol Q_h,\\
(\boldsymbol p_h, \nabla_h v_h)- \langle \bar p_h(u_h,\bm p_h)+\check
p_h, [v_h]_e\rangle &= -(f, v_h)+\langle g_N,
v_h\rangle_{\Gamma_N}~\,\quad \forall v_h\in V_h,\\
-\langle [u_h]_e-\tau^{-1} \check
p_h,\check{q}_h\rangle &= -\langle
g_D,\check{q}_h\rangle_{\Gamma_D}  ~~\qquad\qquad\quad \forall \check{q}_h \in
\check{Q}_h, \\
\langle [\bm p_h]- \eta^{-1} \check
u_h,\check{v}_h\rangle &= \langle
g_N,\check{v}_h\rangle_{\Gamma_N} ~\,\qquad\qquad\qquad \forall \check{v}_h \in
\check{V}_h.
\end{aligned}
\right. 
\end{equation}
We point out here that if $\check Q_h=\{0\}, \check V_h=\{0\}$, then
the above method \eqref{eq:symm-ABCMDG} induce to the consistent
methods listed in \cite[Table 3.1]{arnold2002unified}.

\paragraph{Compact form for a special case.} In what follows, in this
paper, we consider a special case: $\bar u_h(u_h)=\{u_h\}$ and $\bar
p_h(u_h,\bm p_h)=\{\bm p_h\}_e$. In this case, the formulation
\eqref{eq:symm-ABCMDG} can be recast into the following compact form:
 Find $(\bm p_h, \check p_h, u_h, \check
u_h)\in \bm Q_h\times\check Q_h\times V_h\times \check V_h$ such
that for any $(\bm q_h, \check q_h,v_h,  \check v_h)\in \bm
Q_h\times\check Q_h\times V_h\times \check V_h$
\begin{equation} \label{eq:XG} 
\left\{
\begin{aligned}
a(\tilde{\bm p}_h, \tilde{\bm q}_h) + b(\tilde{\bm
q}_h, \tilde{u}_h) &= -\langle g_D, \bm q_h\cdot\bm n+\check
q_h\rangle_{\Gamma_D} \qquad \qquad\quad \forall \tilde{\bm q}_h \in
\tilde{\bm Q}_h:=\bm{Q}_h \times \check{Q}_h, \\
b(\tilde{\bm p}_h, \tilde{v}_h) - c(\tilde{u}_h,
\tilde{v}_h)
&= -(f, v_h)+\langle g_N,v_h+\check{v}_h\rangle_{\Gamma_N}  ~\qquad
\forall \tilde{v}_h \in \tilde{V}_h:=V_h \times \check{V}_h, 
\end{aligned}
\right.
\end{equation} 
where $\tilde{\bm{p}}_h := (\bm{p}_h, \check{p}_h)$, $\tilde{u}_h :=
(u_h, \check{u}_h)$ and 
\begin{subequations} 
\begin{align}
a(\tilde{\bm p}_h, \tilde{\bm q}_h) &:=
(c\bm{p}_h, \bm{q}_h) + \langle\tau^{-1}\check p_h, \check q_h
\rangle, \label{eq:XG-forms-a} \\
 c(\tilde{u}_h, \tilde{v}_h) &:= \langle \eta^{-1}\check u_h, \check v_h\rangle, \label{eq:XG-forms-c}\\
 b(\tilde{\bm q}_h, \tilde{u}_h) &:= (\nabla_h u_h, \bm{q}_h)
-\langle[u_h]_e, \{\bm q_h\}_e\rangle+\langle\check u_h, [\bm q_h]\rangle - \langle
 [u_h]_e, \check{q}_h \rangle, \label{eq:XG-forms-b1}\\
&:=-(u_h, {\rm div}_h \bm q_h)+\langle\{u_h\}, [\bm q_h]\rangle+\langle\check u_h, [\bm q_h]\rangle - \langle
 [u_h]_e, \check{q}_h \rangle, \label{eq:XG-forms-b2}
\end{align}
\end{subequations}
where \eqref{eq:dg-identity} is used to rewrite the bilinear form $
b(\tilde{\bm q}_h, \tilde{u}_h)$.
\begin{remark}
We note that if $(\bm p,  u)$ is the solution of \eqref{H1}, then
$(\bm p, 0; u, 0)$ satisfies the equations \eqref{eq:XG}. Namely, the
formulation \eqref{eq:XG} is consistent. 
\end{remark}
Let 
\begin{equation} \label{eq:big-bilinear}
\tilde{a}((\tilde {\bm p}_h, \tilde u_h), (\tilde{\bm q}_h,
\tilde v_h)) := a(\tilde{\bm p}_h, \tilde{\bm q}_h) + b(\tilde{\bm q}_h, \tilde{u}_h) +  b(\tilde{\bm p}_h,
\tilde{v}_h) - c(\tilde{u}_h, \tilde{v}_h).
\end{equation}
Motivated by the two formulations of $b(\tilde{\bm q}_h,
\tilde{u}_h)$ in \eqref{eq:XG-forms-b1} and \eqref{eq:XG-forms-b2}, we
have two types of inf-sup conditions for the formulation \eqref{eq:XG}, which will be
discussed in next section.

\section{Unified Analysis of the Four Filed Formulation} \label{sec:well-posedness}
In this section, we shall present two types of the inf-sup condition
for the formulation \eqref{eq:XG}.   
\subsection{Gradient-based uniform inf-sup condition}
\label{subsec:grad}

Let us consider the well-posedness of formulation \eqref{eq:XG} in
the gradient-based case.  For any $\bm p_h\in \bm Q_h, \check p_h\in
\check Q_h, u_h\in V_h, \check u_h\in \check V_h$, define
\begin{equation}\label{eq:grad-norms}
\begin{aligned}
\|\tilde{\bm p}_h\|^2_{0,\rho_h}&:= \underbrace{(c
\bm{p}_h,\bm{p}_h)}_{\|{\bm p}_h\|^2_{0,c}} + \underbrace{
\langle \rho h_e\check p_h, \check p_h\rangle}_ {\|\check
p_h\|_{0,\rho_h}^2}, \\
\|\tilde u_h\|^2_{1,\rho_h}& := \underbrace{
(\nabla_h u_h,\nabla_h u_h) + \langle \rho^{-1} h^{-1}_e\check
{\mathcal Q}_h^p [u_h]_e, \check {\mathcal Q}_h^p
[u_h]_e\rangle}_{\|u_h\|^2_{1,\rho_h}} 
+ \underbrace{ \langle  \rho^{-1} h^{-1}_e\check u_h, \check
u_h\rangle}_{\|\check u_h\|^2_{0,\rho_h^{-1}}},
\end{aligned}
\end{equation}
where $\check{\mathcal{Q}}_h^{p}: L^2(\mathcal E_h)\rightarrow
\check{Q}_h$ and $\check{\mathcal{Q}}_h^u: L^2(\mathcal E_h)
\rightarrow \check{V}_h$ are the $L^2$ projections. Here, we
abbreviate the dependence of both $\rho$ and mesh size $h$ in the
norms to $\rho_h$.

We are now ready to state the first main result. 

\begin{theorem} \label{thm:grad-infsup}
If we choose $\tau = (\rho h_e)^{-1}, \eta\cong\tau^{-1} = \rho h_e$ in
 formulation \eqref{eq:XG} and the spaces $ \bm Q_h,
\check{Q}_h, V_h$ satisfy the conditions:
\begin{enumerate}
\item [(a)] $\check{Q}_h$ contains piecewise constant function space;
\item [(b)] $\nabla_h V_h \subset \bm Q_h$;
\item [(c)] $\{\nabla_h V_h\}_e \subset \check{Q}_h$. 
\end{enumerate}
Then we have:
\begin{enumerate}
\item  There exists $\rho_0 > 0$ such that $\tilde{a}((\cdot,\cdot),(\cdot,\cdot))$ in \eqref{eq:big-bilinear} is uniformly well-posed with respect
to the norms $\|\cdot\|_{0,\rho_h}$, $\|\cdot\|_{1,\rho_h}$ when $\rho
\in (0, \rho_0]$ and the following estimates holds:
\begin{equation}\label{eq:grad-stability}
\|\bm{p}_h\|_{0,c}+\|\check p_h\|_{0,\rho_h}+\|u_h\|_{1,\rho_h}+\|\check
u_h\|_{0,\rho_h^{-1}} 
 \lesssim \|f\|_{-1,\rho_h} + \|g_D\|_{\frac{1}{2},\rho_h,
   \Gamma_D} + \|g_N\|_{-\frac{1}{2}, \rho_h, \Gamma_N},
\end{equation} 
where 
$$ 
\begin{aligned}
\|f\|_{-1,\rho_h}&:= \sup\limits_{v_h\in V_h \setminus\{0\}} \frac{(f,
  v_h)}{\|v_h\|_{1,\rho_h}}, \\
\|g_D\|_{\frac{1}{2},\rho_h,\Gamma_D} &:= \sup\limits_{\bm q_h\in \bm
Q_h \setminus \{\bm{0} \}} \frac{(g_D, \bm q_h\cdot \bm n)_{\Gamma_D}}{\|\bm
q_h\|_{0,c}} + 
\sup\limits_{\check q_h\in \check Q_h \setminus\{0\}} \frac{(g_D,
  \check q_h)_{\Gamma_D}}{\|\check q_h\|_{0,\rho_h}}, \\
\|g_N\|_{-\frac{1}{2},\rho_h,\Gamma_N} &:= \sup\limits_{v_h\in V_h
\setminus \{0\}} \frac{(g_N, v_h)_{\Gamma_N}}{\|v_h\|_{1, \rho_h}}
+ \sup\limits_{\check v_h\in \check V_h \setminus \{0\}} \frac{(g_N,
  \check v_h)_{\Gamma_N}}{\|\check v_h\|_{0,\rho_h^{-1}}}.
\end{aligned}
$$ 
\item Let $(\bm p,  u)\in \bm {L}^2(\Omega)\times
H^1(\Omega) $ be the solution of \eqref{H1} and $(\tilde{\bm p}_h,
\tilde{u}_h)\in \tilde{\bm Q}_h\times \tilde{V}_h$ be the solution of
\eqref{eq:XG}, we have the quasi-optimal approximation as follows: 
\begin{equation} \label{eq:grad-quasi}
\|\bm p-\bm{ p}_h\|_{0,c} + \|\check p_h\|_{0,\rho_h} + \|u-
u_h\|_{1,\rho_h}+\|\check u_h\|_{0,\rho_h^{-1}}\lesssim \inf\limits_{\bm{
q}_h\in \bm Q_h, v_h\in
V_h}\left(\|\bm p-\bm{ q}_h\|_{0,c}+\|u- v_h\|_{1,\rho_h}\right).
\end{equation}
\item If $\bm p\in \bm {H}^{k+1}(\Omega), u\in H^{k+2}(\Omega)$ $
(k\ge 0)$ and we choose the spaces $\bm Q_h\times \check Q_h\times
V_h\times \check V_h=\bm Q_h^{k}\times \check Q_h^k\times
V_h^{k+1}\times\check V_h$ for any $\check V_h$, then we have the
error estimate
\begin{equation} \label{eq:grad-estimate}
\|\bm p-\bm{p}_h\|_{0,c} + \|\check p_h\|_{0,\rho_h}+ \|u-
 u_h\|_{1,\rho_h}+\|\check u_h\|_{0,\rho_h^{-1}}\lesssim h^{k+1} (|\bm
   p|_{k+1}+|u|_{k+2}).
\end{equation}
\end{enumerate}
\end{theorem}

\begin{proof}
First, we consider the boundedness of formulation 
\eqref{eq:big-bilinear}. The boundedness of $a(\cdot,\cdot)$ and
$c(\cdot,\cdot)$ follows directly from the definition of
parameter-dependent norms \eqref{eq:grad-norms}. For the boundedness
of $b(\cdot, \cdot)$ given in \eqref{eq:XG-forms-b1}, we use the
Cauchy-Schwarz inequality, trace inequality, and the following
inequality 
$$ 
h_e^{-1}\| [u_h]_e\|_{0,e}^2 
\lesssim |\nabla_h u_h|_{0,
\omega_e}^2 +  h_e^{-1} \|\check{\mathcal Q}_h^p [u_h]_e\|^2_{0,e},
$$ 
provided that $\check{Q}_h$ contains piecewise constant function
space. Here, $\omega_e=\bigcup\limits_{e\subset \partial K}K$.

Next we consider the inf-sup condition for the bilinear form
$\tilde{a}((\cdot,\cdot),(\cdot,\cdot))$ defined in
\eqref{eq:big-bilinear}. The proof follows from the technique shown in
\cite{hong2017uniformly}.  For any given $(\tilde{\bm p}_h, \tilde
u_h)$, since $\nabla_h V_h \subset \bm Q_h$, we choose
\begin{equation} \label{eq:grad-based-qv}
\tilde{\bm q}_h = \gamma \tilde{\bm p}_h +
\tilde{\bm s}_h := \gamma \tilde{\bm p}_h + 
\begin{pmatrix}
\nabla_h u_h \\
-\rho^{-1}h_e^{-1}\check {\mathcal Q}_h^p [u_h]_e
\end{pmatrix}, \qquad \tilde v_h=-\gamma \tilde
u_h,
\end{equation}
where $\gamma$ is a constant that will be determined later. The
boundedness of $\tilde{\bm q}_h$ and $\tilde v_h$ under the
parameter-dependent norms \eqref{eq:grad-norms} is straightforward.
Next, we have 
$$ 
\begin{aligned}
\tilde{a}((\tilde{\bm p}_h, \tilde u_h),
(\tilde{\bm q}_h, \tilde v_h)) 
&= a( \tilde{\bm p}_h,
\gamma \tilde{\bm p}_h + \tilde{\bm s}_h ) + b( \gamma
\tilde{\bm p}_h + \tilde{\bm s}_h, \tilde{u}_h) + 
b(\tilde{\bm p}_h, -\gamma \tilde{u}_h) + \gamma
c(\tilde{u}_h, \tilde{u}_h) \\
&= \gamma a(\tilde{\bm p}_h, \tilde{\bm p}_h) + a(\tilde{\bm p}_h, \tilde{\bm s}_h) + b(\tilde{\bm s}_h,
\tilde{u}_h) + \gamma c(\tilde{u}_h, \tilde{u}_h) \\
&= \gamma \|\tilde{\bm p}_h\|_{0,\rho_h}^2 + \gamma \langle \eta^{-1}\check{u}_h, \check{u}_h \rangle
 + a(\tilde{\bm p}_h, \tilde{\bm s}_h) + b(\tilde{\bm s}_h, \tilde{u}_h) \\ 
&\geq \gamma \|\tilde{\bm p}_h\|_{0,\rho_h}^2 + C_0\gamma \|\check{u}_h\|_{0, \rho_h^{-1}}^2
 + a(\tilde{\bm p}_h, \tilde{\bm s}_h) + b(\tilde{\bm s}_h, \tilde{u}_h).
\end{aligned}
$$ 
Clearly, from the definitions of $a(\cdot,\cdot)$  in
\eqref{eq:XG-forms-a} and $b(\cdot,\cdot)$ in
\eqref{eq:XG-forms-b1}, we have  
$$ 
\begin{aligned}
a(\tilde{\bm p}_h, \tilde{\bm s}_h) &\geq
-\epsilon_1\|\tilde{\bm s}_h\|_{0,\rho_h}^2 
-\epsilon_1^{-1} \|\tilde{\bm p}_h\|_{0,\rho_h}^2\\
& = -\epsilon_1 (c\nabla_h u_h, \nabla_h u_h) - \epsilon_1
\langle \rho^{-1} h_e^{-1} \check{\mathcal Q}_h^p [u_h]_e,
 \check{\mathcal Q}_e^p [u_h]_e \rangle-\epsilon_1^{-1} \|\tilde{\bm p}_h\|_{0,\rho_h}^2,\\
b(\tilde{\bm s}_h, \tilde{u}_h) & = \|\nabla_h u_h\|_0^2 +
\langle \rho^{-1} h_e^{-1} \check{\mathcal Q}_h^p  [u_h]_e,
  \check{\mathcal Q}_h^p [u_h]_e\rangle + \langle\check
  u_h, [\nabla_h u_h]\rangle  - \langle\llbracket u_h
  \rrbracket, \{\nabla_h u_h\}\rangle. 
\end{aligned}
$$ 
The standard Cauchy-Schwarz inequality, trace inequality, inverse
inequality and the third condition $\{\nabla_hV_h\}_e \subset \check
Q_h$ imply that 
$$ 
\begin{aligned}
\langle\check u_h, [\nabla_h u_h]\rangle &
\ge-\epsilon_2^{-1}\sum_{e\in \mathcal E_h}\|h_e^{-\frac{1}{2}}\check u_h\|^2_{0,e}-
\epsilon_2\sum_{e\in \mathcal E_h}\|h_e^{\frac{1}{2}}[\nabla_h u_h]\|^2_{0,e}\\
&\geq -\epsilon_2^{-1} \langle h^{-1}_e\check u_h, \check
u_h\rangle-C_1\epsilon_2\sum_{e\in \mathcal E_h}
\|\nabla_h u_h\|_{0,\omega_e}^2 \\
&\geq - \rho\epsilon_2^{-1} \|\check u_h\|_{0,\rho_h^{-1}}^2
 -C_2\epsilon_2 \|\nabla_h u_h\|_{0}^2, \\
-\langle\llbracket u_h \rrbracket, \{\nabla_h u_h\}\rangle& \ge -\epsilon_3\|\nabla_hu_h\|_{0}^2 -C_3\epsilon_3^{-1}\langle h_e^{-1}
\check{\mathcal Q}_h^p [u_h]_e, \check{\mathcal
  Q}_h^p [u_h]_e \rangle.
\end{aligned}
$$ 
Therefore, from the above inequalities, we deduce that when $\rho \in
(0, \rho_0]$,
$$ 
\begin{aligned}
\tilde{a}((\tilde{\bm p}_h, \tilde u_h), (\tilde{\bm
q}_h, \tilde v_h)) &\geq (\gamma - \epsilon_1^{-1}) \|\tilde{\bm
p}_h\|_{0,\rho_h}^2 + (C_0\gamma -\rho\epsilon_2^{-1}) \|\check{u}_h\|_{0,\rho_h^{-1}}^2 \\
& ~~~~+(1-\|c\|_{\infty}\epsilon_1- C_2\epsilon_2-\epsilon_3)\|\nabla_h u_h\|_0^2 +
(1-\epsilon_1-C_3\rho\epsilon_3^{-1})\langle \rho^{-1}  h_e^{-1}
\check{\mathcal Q}_h^p [u_h]_e, \check{\mathcal Q}_h^p  [u_h]_e
\rangle \\
& \geq \frac{1}{4} \left( \|\tilde{\bm p}_h\|_{0,\rho_h}^2 +
\|\tilde{u}_h\|_{1, \rho_h}^2 \right), 
\end{aligned}
$$ 
by choosing $\epsilon_1$, $\epsilon_2,\epsilon_3$, $\gamma$ and
$\rho_0$ as 
$$ 
\epsilon_1 = \frac{1}{4\max\{\|c\|_{\infty}, 1\}}, \quad 
\epsilon_2 = \frac{1}{4C_2}, \quad 
\epsilon_3 = \frac{1}{4}, \quad 
\gamma = \frac{1}{4} + \frac{1}{2C_0} + 4\max\{\|c\|_{\infty}, 1\}, \quad 
\rho_0 = \min\{ \frac{1}{16C_2}, \frac{1}{8C_3}\}.
$$ 
Hence, we have the inf-sup condition for $\tilde{a}((\cdot,\cdot),
(\cdot,\cdot))$ under the parameter-dependent norms
\eqref{eq:grad-norms}. The stability result \eqref{eq:grad-stability},
quasi-optimal error estimates \eqref{eq:grad-quasi} and
\eqref{eq:grad-estimate} then follow directly from the Babu\v{s}ka
theory and interpolation theory. 
\end{proof}

\subsection{Divergence-based uniform inf-sup condition}
In light of the formulation of $b(\cdot, \cdot)$ in
\eqref{eq:XG-forms-b2}, we then establish the divergence-based inf-sup
condition. For any $\bm p_h\in \bm Q_h, \check p_h\in \check Q_h,
u_h\in V_h, \check u_h\in \check V_h$, the norms are defined by 
\begin{equation}\label{eq:div-norms}
\begin{aligned}
\|\tilde{\bm p}_h\|^2_{{\rm div},\rho_h} &:= 
\underbrace{ 
(c \bm{p}_h, \bm{p}_h) +
({\rm div}_h\bm{p}_h, {\rm div}_h\bm{p}_h)
 + \langle\rho^{-1} h_e^{-1}
\check{\mathcal Q}_h^u [\bm{p}_h],
\check{\mathcal Q}_h^u[\bm{p}_h] \rangle
}_{\|{\bm p}_h\|^2_{{\rm div},\rho_h}} +
\underbrace{ \langle \rho^{-1} h_e^{-1}\check p_h,\check
  p_h\rangle }_{\|\check p_h\|_{0,\rho_h^{-1}}^2},\\
\|\tilde u_h\|^2_{0,\rho_h} &:=
\underbrace{ (u_h,u_h)}_{\|u_h\|_0^2} +  
\underbrace{  \langle\rho h_e\check u_h,\check u_h\rangle}_{\|\check u_h\|^2_{0,\rho_h}}.
\end{aligned}
\end{equation}
We are now in the position to state the second main result. 

\begin{theorem} \label{thm:div-infsup}
If we choose $\eta= (\rho h_e)^{-1}, \tau\cong\eta^{-1} = \rho h_e$ in
the formulation \eqref{eq:XG} and the spaces $ \bm Q_h,
V_h,\check{V}_h$ satisfy the conditions
\begin{enumerate}
\item [(a)] Let $\bm R_h := \bm Q_h \cap \bm{H}({\rm div},\Omega)$ and
$\bm R_h \times V_h$ is a stable pair for mixed method;
\item [(b)] ${\rm div}_h \bm Q_h = V_h$;
\item [(c)] $\{{\rm div}_h \bm Q_h\} \subset \check{V}_h$.
\end{enumerate}
Then we have 
\begin{enumerate}
\item There exists $\rho_0 > 0$, $\tilde{a}((\cdot,\cdot),
(\cdot,\cdot))$ in \eqref{eq:big-bilinear} is uniformly well-posed with respect to the norms
$\|\cdot\|_{{\rm div}, \rho_h}$, $\|\cdot\|_{0,\rho_h}$ when $\rho\in
(0, \rho_0]$ and the following estimate holds:
\begin{equation}\label{eq:div-stability}
\|\bm{ p}_h\|_{{\rm div},\rho_h}+\|\check p_h\|_{0,\rho_h^{-1}}+\|u_h\|_0+\|\check
u_h\|_{0,\rho_h}\lesssim \|f\|_0 +\|g_D\|_{-\frac{1}{2},\rho_h,\Gamma_D} + 
\|g_N\|_{\frac{1}{2}, \rho_h,\Gamma_N}.
\end{equation} 
where 
$$ 
\begin{aligned}
\|g_D\|_{-\frac{1}{2}, \rho_h, \Gamma_D} &:= 
\sup\limits_{\bm q_h\in \bm Q_h \setminus\{\bm{0}\}} \frac{(g_D, \bm
  q_h\cdot \bm n)_{\Gamma_D}}{\|\bm q_h\|_{{\rm div}, \rho_h}} + 
\sup\limits_{\check q_h\in \check Q_h \setminus \{0\}} \frac{(g_D,
  \check q_h)_{\Gamma_D}}{\|\check q_h\|_{1,\rho_h}}, \\
\|g_N\|_{\frac{1}{2},\rho_h,\Gamma_N}& :=\sup\limits_{v_h\in V_h
\setminus \{0\}}
\frac{(g_N, v_h)_{\Gamma_N}}{\|v_h\|_0}
+ \sup\limits_{\check v_h\in \check V_h \setminus\{0\}} \frac{(g_N, \check
  v_h)_{\Gamma_N}}{\|\check v_h\|_{0,\rho_h}}.
\end{aligned}
$$ 
\item Let $(\bm p,  u)\in H({\rm div}, \Omega)\times
L^2(\Omega)$ be the solution of \eqref{H1} and $(\tilde{\bm p}_h,
\tilde{u}_h)\in \tilde{\bm Q}_h\times \tilde{V}_h$
be the solution of \eqref{eq:XG}, we have the following quasi-optimal
approximation: 
\begin{equation} \label{eq:div-quasi}
 \|\bm p-\bm{ p}_h\|_{{\rm div},\rho_h} 
 + \|\check p_h\|_{0,\rho_h^{-1}} + \|u- u_h\|_0 +\|\check u_h\|_{0,\rho_h} 
 \lesssim \inf_{\bm q_h\in \bm Q_h, v_h\in V_h}\left(\|\bm
     p-\bm{q}_h\|_{{\rm div},\rho_h}+\|u- v_h\|_0\right).
\end{equation}
\item If $\bm p\in \bm {H}^{k+2}(\Omega), u\in H^{k+1}(\Omega) ~(k\ge 0)$, and we choose the
spaces $\bm Q_h\times \check Q_h\times V_h\times \check V_h=\bm
Q_h^{k,RT} (or~\bm{Q}_h^{k+1}) \times \check Q_h \times V_h^{k}\times\check V_h^{k}$
for any $\check Q_h$, then the following estimate holds:
\begin{equation} \label{eq:div-estimate}
 \|\bm p-\bm{ p}_h\|_{{\rm div},\rho_h}+\|\check p_h\|_{0,\rho_h^{-1}}
 +\|u- u_h\|_0+\|\check u_h\|_{0,\rho_h}\lesssim h^{k+1} (|\bm p|_{k+2}+|u|_{k+1}).
\end{equation}
\end{enumerate}
\end{theorem}
\begin{proof}
Since $\{V_h\}=\{ {\rm div}_h \bm Q_h\}\subset \check V_h$, the
boundedness of $\tilde{a}((\cdot,\cdot),(\cdot,\cdot)) $ under
the divergence-based norms \eqref{eq:div-norms} is standard (by
the Piola's transformation) and is therefore omitted. 

Next we consider the inf-sup condition for the bilinear form
$\tilde{a}((\cdot,\cdot),(\cdot,\cdot))$ defined in
\eqref{eq:big-bilinear}.  The proof follows from the technique shown
in \cite{hong2017uniformly}.  Since $\bm R_h \times V_h$ is a stable
pair for mixed method, for any given $(\tilde{\bm p}_h, \tilde u_h)$,
there exists $\bm r_h\in \bm R_h$ such that 
\begin{equation}\label{eq:RT0}
-{\rm div} \bm r_h=u_h \quad \hbox{and} \quad \|\bm
r_h\|_0 + \|{\rm div} \bm r_h\|_0 \leq C_{\rm stab} \|u_h\|_0.
\end{equation}
Now, we choose
\begin{equation}
\tilde{\bm q}_h = \gamma \tilde{\bm p}_h +
\tilde{\bm s}_h := \gamma \tilde{\bm p}_h + 
\begin{pmatrix}
\bm r_h \\
0
\end{pmatrix},\qquad
\tilde{v}_h=-\gamma \tilde{u}_h - \tilde{w}_h := -\gamma \tilde{u}_h - 
\begin{pmatrix}
{\rm div}_h \bm p_h \\
-\rho^{-1}h_e^{-1}\check{\mathcal Q}_h^u [\bm p_h]
\end{pmatrix},
\end{equation}
where $\gamma$ is a constant which will be determined later. First, we
have the boundedness of $\tilde{\bm q}_h$ and $\tilde{v}_h$ by using
\eqref{eq:RT0}, 
$$ 
\begin{aligned}
\|\tilde{\bm q}_h\|_{{\rm div},\rho_h} &\leq \gamma \|\tilde{\bm
  p}_h\|_{{\rm div},\rho_h} + \|\tilde{\bm s}_h\|_{{\rm div},\rho_h}
\lesssim \|\tilde{\bm p}_h\|_{{\rm div}, \rho_h} + \|u_h\|_0, \\
\|\tilde{v}_h\|_{0,\rho_h} &\leq \gamma\|\tilde{u}_h\|_{0,\rho_h} +
\|\tilde{w}_h\|_{0,\rho_h} \lesssim  \|\tilde{u}_h\|_{0,\rho_h} +
\|{\rm div}_h \bm p_h \|_0 +  \langle \rho^{-1}h_e^{-1}
\check{\mathcal Q}_h^u[\bm p_h], \check{\mathcal Q}_h^u[\bm p_h] 
\rangle^{1/2}.
\end{aligned}
$$ 
Next, we have 
$$ 
\begin{aligned}
\tilde{a}((\tilde{\bm p}_h, \tilde u_h),
(\tilde{\bm q}_h, \tilde v_h)) 
&= a( \tilde{\bm p}_h, \gamma \tilde{\bm p}_h + \tilde{\bm
s}_h ) + b(\gamma \tilde{\bm p}_h + \tilde{\bm s}_h, \tilde{u}_h) +
b(\tilde{\bm p}_h, -\gamma \tilde{u}_h - \tilde{w}_h) +
c(\tilde{u}_h, \gamma \tilde{u}_h + \tilde{w}_h) \\
&= \gamma a(\tilde{\bm p}_h, \tilde{\bm p}_h) +
a(\tilde{\bm p}_h, \tilde{\bm s}_h) + 
b(\tilde{\bm s}_h, \tilde{u}_h) - b(\tilde{\bm
p}_h, \tilde{w}_h) + \gamma c(\tilde{u}_h, \tilde{u}_h) +
c(\tilde{u}_h, \tilde{w}_h)\\
&= \gamma (c\bm{p}_h, \bm{p}_h)+ \gamma \langle \tau^{-1}\check p_h, \check p_h\rangle+ \gamma 
\langle\rho h_e\check u_h,\check u_h\rangle \\
& ~~ + a(\tilde{\bm p}_h, \tilde{\bm s}_h) + b(\tilde{\bm s}_h, \tilde{u}_h) - b(\tilde{\bm p}_h,
\tilde{w}_h) + c(\tilde{u}_h, \tilde{w}_h) \\
&\geq \gamma (c\bm{p}_h, \bm{p}_h)+ C_0\gamma \|\check p_h\|_{0,\rho_h^{-1}}^2 + \gamma 
\|\check u_h\|_{0,\rho_h}^2\\
& ~~ + a(\tilde{\bm p}_h, \tilde{\bm s}_h) + b(\tilde{\bm s}_h, \tilde{u}_h) - b(\tilde{\bm p}_h,
\tilde{w}_h) + c(\tilde{u}_h, \tilde{w}_h).
\end{aligned}
$$
Clearly, from the definitions of $a(\cdot,\cdot)$ in
\eqref{eq:XG-forms-a} and $c(\cdot,\cdot)$ in
\eqref{eq:XG-forms-c}, we have 
$$ 
\begin{aligned}
a(\tilde{\bm p}_h, \tilde{\bm s}_h) &\geq -\epsilon_1 \|\bm
r_h\|_0^2 - C_1\epsilon_1^{-1}  (c\bm p_h, \bm p_h) 
\geq -C_{\rm stab}^2\epsilon_1 \|u_h\|_0^2 - C_1\epsilon_1^{-1}(c\bm p_h, \bm
    p_h),\\
c(\tilde{u}_h, \tilde{w}_h) &\geq -\epsilon_2 
\langle\rho^{-1}h_e^{-1}\check{\mathcal Q}_h^u[\bm p_h],\check{\mathcal Q}_h^u[\bm p_h]\rangle -
C_2\epsilon_2^{-1} \|\check u_h\|_{0,\rho_h}^2. 
\end{aligned}
$$ 
Further, from \eqref{eq:RT0} and the formulation of $b(\cdot,\cdot)$
in \eqref{eq:XG-forms-b2}, we have 
$$ 
\begin{aligned}
b(\tilde{\bm s}_h, \tilde{u}_h) &= \|u_h\|_0^2, \\
-b(\tilde{\bm p}_h, \tilde{w}_h) &= \|{\rm div}_h \bm
p_h\|_0^2 + \langle\rho^{-1}h_e^{-1}\check{\mathcal Q}_h^u[\bm p_h],\check{\mathcal Q}_h^u[\bm p_h]\rangle
 + \langle [{\rm div}_h \bm p_h]_e, \check
p_h\rangle - \langle \{{\rm div}_h \bm p_h\}, [\bm p_h]
\rangle \\
& \geq (1-\epsilon_3-\epsilon_4)\|{\rm div}_h \bm p_h\|_0^2 
- C_3\rho \epsilon_3^{-1} \|\check p_h\|_{0,\rho_h^{-1}}^2 
+ (1-C_4\rho\epsilon_4^{-1}) \langle \rho^{-1}h_e^{-1}\check{\mathcal Q}_h^u[\bm p_h],  
\check{\mathcal Q}_h^u[\bm p_h] \rangle.
\end{aligned}
$$ 
Here, we use the Cauchy-Schwarz inequality, trace inequality and the
fact that $\{{\rm div}_h \bm{Q}_h\} \subset \check V_h$ in the last
step.  Therefore, from the above inequalities, we deduce that then
when $\rho \in (0, \rho_0]$,
$$ 
\begin{aligned}
\tilde{a}((\tilde{\bm p}_h, \tilde{u}_h), (\tilde{\bm q}_h,
\tilde{v}_h)) & \geq (\gamma - C_1\epsilon_1^{-1}) (c\bm{p}_h,
\bm{p}_h) + (1-\epsilon_3 -\epsilon_4)\|{\rm div}_h \bm p_h\|_0^2 \\ 
&~~~+ (1-\epsilon_2-C_4\rho\epsilon_4^{-1})
\langle \rho^{-1}  h_e^{-1}\check{\mathcal Q}_h^u[\bm p_h], 
\check{\mathcal Q}_h^u[\bm p_h]\rangle \\
&~~~+ (C_0\gamma - C_3\rho\epsilon_3^{-1})  \|\check
p_h \|_{0,\rho_h^{-1}}^2\\
&~~~+ (1-C_{\rm stab}^2\epsilon_1)\|u_h\|_0^2 
+ (\gamma - C_2\epsilon_2^{-1}) \|\check u_h\|_{0,\rho_h}^2 \\
& \geq \frac{1}{2}\left( \|\tilde{\bm p}_h\|_{{\rm div},\rho_h}^2 + \|\tilde
   v_h\|_{0,\rho_h}^2 \right),
\end{aligned}
$$ 
by choosing $\epsilon_1$, $\epsilon_2$, $\epsilon_3, \epsilon_4$,
$\gamma$ and $\rho_0$ as 
$$ 
\epsilon_1 = \frac{1}{2C_{\rm stab}^2},\quad \epsilon_2 = \epsilon_3 =
\epsilon_4 = \frac{1}{4}, 
\quad \gamma = \frac{1}{2} + \frac{1}{2C_0} + \max\{2C_{\rm stab}^2C_1, 4C_2, \frac{C_3}{4C_0C_4}\},
\quad \rho_0 = \frac{1}{16C_4}.
$$ 
Hence, we have the inf-sup condition for $\tilde{a}((\cdot,\cdot),
(\cdot,\cdot))$ under the parameter-dependent norms
\eqref{eq:div-norms}. The stability result \eqref{eq:div-stability},
quasi-optimal error estimates \eqref{eq:div-quasi} and
\eqref{eq:div-estimate} then follow directly from the Babu\v{s}ka
theory and interpolation theory. 
\end{proof} 

\section{Some limiting case of four filed formulation} \label{sec:limiting}
With the uniform inf-sup conditions, we revisit some limiting of formulation \eqref{eq:XG} in
case of $\rho\to 0$ \cite{hong2017unified}.

First, having the gradient-based inf-sup condition, we discuss the
limiting of formulation \eqref{eq:XG} with $g_D=0$ in case of $\tau = (\rho h_e)^{-1}$, $\eta \cong
\tau^{-1} = \rho h_e$ as $\rho\to 0$. Denote
$H^1_{0,\Gamma_D}(\Omega)=\{u\in H^1(\Omega): u|_{\Gamma_D}=0\}$.
Consider the $H^1$ conforming subspace $V^c_{h,g} = V_h\cap H^1_{0,
\Gamma_D}(\Omega)\subset V_h$, then the primal method when applying to
the Poisson equation \eqref{pois} can be written as: Find $(u^c_h,
\bm p^c_h)\in V^c_{h,g}\times \bm Q_h$ such that  
\begin{equation} \label{primalmethod}
\left\{
\begin{aligned}
(c\bm p^c_h, \bm q_h) + (\nabla u^c_h,\bm q_h)
  &=G_p(\bm q_h)\qquad \forall \bm q_h\in \bm Q_h,\\
(\bm p^c_h, \nabla v^c_h) &=F_p(v^c_h)\qquad\forall v^c_h\in
V^c_{h,0}.
\end{aligned}
\right.
\end{equation}
where $G_p(\bm q_h)=0, F_p(v^c_h) =-(f,v^c_h) + \langle g_N,
v_h^c\rangle_{\Gamma_N}$.  Then, by $\nabla  V^c_h\subset \nabla_h
V_h\subset \bm Q_h$, the well-posedness of the primal method (cf.
\cite{brenner2007mathematical}) implies that  
\begin{equation} \label{primalstability}
 \|\bm{p}^c_h\|_{0,c} + \|{u}^c_h\|_{1}  \leq
 C_p\left( \sup\limits_{\bm q_h\in \bm Q_h \setminus
\{\bm{0}\}}\frac{G_p(\bm q_h)}{\|\bm q_h\|_{0,c}}+ \sup\limits_{v^c_h\in
V_{h,g}^{c} \setminus
\{0\}}\frac{F_p(v^c_{h})}{\|{v}^c_h\|_{1}}\right).
\end{equation}
We have the following theorem.
\begin{theorem}
Assume that the spaces $\bm{Q}_h, V_h$ and $\check Q_h$
satisfy
\begin{enumerate}
\item[(a)] $\nabla_h V_h \subset \bm{Q}_h$;
\item[(b)] $\{\bm Q_h\}_e \subset \check Q_h$; 
\item[(c)] $V_h=V_h^k~(k\ge 1)$.
\end{enumerate}
Then formulation \eqref{eq:XG} with $g_D=0$ and $\tau = (\rho h_e)^{-1},
\eta\cong\tau^{-1} = \rho h_e$ converges to primal method
\eqref{primalmethod} as $\rho \to 0$.  Further, let $(\bm p_h^{\tau},
\check p_h^{\tau}, u_h^{\tau}, \check u_h^{\tau})$ be the solution of
\eqref{eq:XG} and $(\bm p_h^c, u_h^c)$ be the solution of
\eqref{primalmethod}, we have 
\begin{equation} \label{eq:grad-limit}
\|\bm p_h^{\tau}-\bm p_h^c\|_{0,c}
+(\|\nabla_h(u_h^{\tau}-u_h^c)\|_0^2+\sum_{e\in \mathcal
E_{h}}h_e^{-1}\|\llbracket u_h^{\tau}-u_h^c
\rrbracket\|^2_{0,e})^{\frac{1}{2}}\lesssim \rho^{\frac{1}{2}} R_p,
\end{equation}
where
$R_p := \|f\|_{-1,\rho_h}+\|g_N\|_{-\frac{1}{2}, \rho_h, \Gamma_N}$.
\end{theorem}
\begin{proof}
Taking $v_h= v_h^c$ in the second equation in \eqref{eq:symm-DG}, we
see that 
\begin{equation} \label{HDG62}
\left\{
\begin{aligned}
(c\bm p_h^{\tau}, \bm q_h) +(\nabla_hu_h^{\tau}, \bm q_h)-\langle
[u_h^{\tau}]_e, \{\bm q_h\}_e\rangle +\langle \check
u_h^{\tau}, [\bm q_h]\rangle & = -\langle g_D, \bm
q_h\cdot \bm n\rangle_{\Gamma_D} \quad\qquad \forall \bm q_h\in \bm
Q_h,\\
(\bm p_h^{\tau}, \nabla_h v_h^c) &=-(f, v_h^c)+\langle g_N,
  v_h^c\rangle_{\Gamma_N}~~ \forall v_h^c\in V_{h,0}^c.
\end{aligned}
\right.
\end{equation}
Let 
$$ 
\bm{\delta}_{h}^p = \bm{p}_h^\tau - \bm{p}_h^c, \qquad \delta_h^u =
u_h^\tau - u_h^c.
$$
Subtracting \eqref{primalmethod} from the equation \eqref{HDG62}, we
have 
$$
\left\{
\begin{aligned}
(c\bm{\delta}_h^p,\bm q_h) + (\nabla_h \delta_h^u ,\bm q_h) &= \langle
[u_h^{\tau}]_e, \{\bm q_h\}_e\rangle-\langle \check
u_h^{\tau}, [\bm q_h]\rangle -\langle g_D, \bm q_h\cdot
\bm n\rangle_{\Gamma_D}\quad \forall \bm q_h\in \bm Q_h, \\
(\bm{\delta}_h^p,\nabla v_h^c)
&= 0 ~~\qquad\qquad\qquad\qquad\qquad\qquad\qquad\qquad \qquad\forall
v_h^c\in V_{h,0}^c.
\end{aligned}
\right.
$$
By the assumption $\{\bm Q_h\}_e\subset \check Q_h$ and noting that
$u_h^{\tau}$ satisfies \eqref{eq:XG3}, we have 
$$
\left\{
\begin{aligned}
(c\bm{\delta}_h^p,\bm q_h) + (\nabla_h \delta_h^u, \bm q_h) &= \langle
\rho h_e \check p_h^{\tau}, \{\bm q_h\}_e\rangle-\langle \check u_h^{\tau}, [\bm q_h]\rangle \quad
\forall \bm q_h\in \bm Q_h, \\
(\bm{\delta}_h^p,\nabla v_h^c)
&= 0  \qquad\qquad\qquad\qquad\qquad\qquad\forall v_h^c\in
V_{h,0}^c.
\end{aligned}
\right.
$$
Further, for any $u_h^I \in V_{h,g}^c$, we have 
\begin{equation}\label{error}
\left\{
\begin{aligned}
(c\bm{\delta}_h^p, \bm q_h) + (\nabla u_h^{I}-\nabla u^c_h,\bm q_h) &=
\langle \rho h_e \check p_h^{\tau}, \{\bm q_h\}_e\rangle-\langle \check u_h^{\tau}, [\bm q_h]\rangle +
(\nabla u_h^I - \nabla u_h^{\tau},\bm q_h) 
\quad \forall \bm q_h \in \bm Q_h, \\
(\bm{\delta}_h^p, \nabla v_h^c) &= 0 \qquad \qquad \qquad \qquad
\qquad \qquad \qquad \qquad
~\qquad\qquad \forall v_h^c \in V_{h,0}^c.
\end{aligned}
\right.
\end{equation}
By the assumption $\nabla_h V_h \subset \bm{Q}_h$, using
\eqref{primalstability}, trace inequality, inverse inequality and
Cauchy-Schwarz inequality, we obtain
\begin{equation}\label{error1}
\begin{aligned}
\|\bm{\delta}_h^p\|_{0,c}+\|u_h^I-u_h^c\|_1 & \leq
C_p\sup_{\bm q_h\in \bm Q_h \setminus \{0\}}\frac{ \langle \rho h_e
\check p_h^{\tau}, \{\bm q_h\}_e\rangle-\langle
\check u_h^{\tau}, [\bm q_h]\rangle + (\nabla u_h^I
- \nabla u_h^{\tau},\bm q_h)}{\|\bm q_h\|_{0,c}} \\
& \lesssim \|\nabla u_h^I - \nabla_hu_h^\tau\|_{0}
+\rho^{\frac{1}{2}}(\|\check p_h^{\tau}\|_{0,\rho_h}+\|\check u_h^{\tau}\|_{0,\rho_h^{-1}}). 
\end{aligned}
\end{equation}
Therefore, noting that $V_h=V_h^k~(k\ge 1)$, \eqref{error1} and
\eqref{eq:grad-stability} imply that 
$$ 
\begin{aligned}
&\|\bm{\delta}_h^p\|_{0,c}+(\|\nabla_h\delta_h^u\|_0^2 + \sum_{e\in
    \mathcal E_{h}}h_e^{-1}\|\llbracket \delta_h^u
    \rrbracket\|^2_{0,e})^{\frac{1}{2}}\\
& \leq \inf_{u_h^I \in V_{h,g}^c} 
\left( \|\bm{\delta}_h^p\|_{0,c} + \|\nabla_h(u_h^I - u_h^c)\|_{0} +
\big(\|\nabla_h(u_h^{\tau}-u_h^I)\|_0^2+\sum_{e\in
      \mathcal E_{h}}h_e^{-1}\|\llbracket u_h^{\tau}-u_h^I
      \rrbracket\|^2_{0,e}\big)^{\frac{1}{2}} \right) \\
& \lesssim  \rho^{\frac{1}{2}}(\|\check p_h^{\tau}\|_{0,\rho_h}+\|\check u_h^{\tau}\|_{0,\rho_h^{-1}}) +  \inf_{u_h^I \in
  V_{h,g}^c} \big(\|\nabla_h(u_h^{\tau}-u_h^I)\|_0^2+\sum_{e\in \mathcal E_{h}}h_e^{-1}\|\llbracket u_h^{\tau}-u_h^I \rrbracket\|^2_{0,e}\big)^{\frac{1}{2}} \\
& \lesssim  \rho^{\frac{1}{2}}(\|\check p_h^{\tau}\|_{0,\rho_h}+\|\check u_h^{\tau}\|_{0,\rho_h^{-1}}) + \big(\sum_{e\in \mathcal
  E_h} h_e^{-1} \|\llbracket u_h^{\tau} \rrbracket\|^2_{0,e}\big)^{\frac{1}{2}} \\
& \lesssim \rho^{\frac{1}{2}}(\|\check p_h^{\tau}\|_{0,\rho_h}+\|\check u_h^{\tau}\|_{0,\rho_h^{-1}}+\|u_h^{\tau}\|_{1,\rho_h})  
\lesssim \rho^{\frac12}R_p.
\end{aligned}
$$
This completes the proof.
\end{proof}

Next, having the divergence-based inf-sup condition, we discuss the
limiting of formulation \eqref{eq:XG} with $g_N=0$ in case of $\eta = (\rho h_e)^{-1}$, $\tau
\cong \eta^{-1} = \rho h_e$ as $\rho \to 0$.  Denote $\bm
H_{0,\Gamma_N}({\rm div}, \Omega)=\{\bm p\in \bm H({\rm div}, \Omega):
\bm p\cdot \bm n|_{\Gamma_N}=0\}$.  Consider the $\bm H({\rm div})$
conforming subspace $\bm Q^c_{h,g}:= \bm Q_h\cap \bm
H_{0,\Gamma_N}({\rm div}, \Omega)\subset \bm Q_h$, the mixed method
when applying to the Poisson equation \eqref{pois} can be written as:
Find $(\bm p^c_h, u^c_h) \in \bm Q^c_{h,g} \times V_h$ such that  
\begin{equation}\label{mixedmethod}
\left\{
\begin{aligned}
(c\bm p^c_h,\bm q^c_h) - (u^c_h,{\rm div} \bm q^c_h) &= G_m(\bm
q_h^c)~\qquad \forall \bm q^c_h\in \bm Q^c_{h,0},\\
 -({\rm div} \bm p^c_h, v_h)&= F_m(v_h) ~\qquad
   \forall v_h\in V_h,
\end{aligned}
\right.
\end{equation}
where $G_m(\bm q_h)=-\langle g_D, \bm q_h\cdot \bm
n\rangle_{\Gamma_D}, F_m(v_h) =-(f,v_h)$.
Then, by the fact that $ {\rm div} \bm Q^c_h= {\rm div}_h \bm
Q_h=V_h$, the well-posedness of the mixed method (cf.
\cite{brezzi1991mixed, boffi2013mixed}) implies that  
\begin{equation} \label{equ:well-posedness-mixed}
\|\bm p_h^c\|_{\bm H({\rm div})} + \|v_h^c\|_0 \lesssim
\sup_{\bm q_h^c \in \bm Q_{h,0}^c \setminus \{\bm 0\}}
\frac{G_m(\bm q_h^c)}{\|\bm
q_h^c\|_{\bm H({\rm div})}} + \sup_{v_h \in V_h \setminus \{0\}}
\frac{F_m(v_h)}{\|v_h\|_0}.
\end{equation}
We have the following theorem.

\begin{theorem}
Assume that the spaces $\bm{Q}_h$, $\check{V}_h$
and $V_h$ satisfy
\begin{enumerate}
\item [(a)] ${\rm div}_h \bm{Q}_h = V_h$;
\item [(b)] $\{V_h\} \subset \check V_h$; 
\item [(c)] $\bm Q_h=\bm Q_h^{k,RT}$ or $\bm Q_h^{k+1}$, $k\ge 0$.
\end{enumerate}
Then formulation \eqref{eq:XG} with $g_N=0$ and $\eta = (\rho h_e)^{-1}$, $\tau \cong
\eta^{-1} = \rho h_e$ converges to mixed method \eqref{mixedmethod} as
$\rho\to 0$. 
Further, let $(\bm p_h^{\eta}, \check p_h^{\eta}, u_h^{\eta}, \check
u_h^{\eta})$ be the solution of \eqref{eq:XG} and $(\bm p_h^c, u_h^c)$
be the solution of \eqref{mixedmethod}, we have 
\begin{equation}
\|\bm p_h^{\eta}-\bm p_h^c\|_{0,c}+\|{\rm div}_h(\bm p_h^{\eta}-\bm
p_h^c)\|_{0}+\|u_h^{\eta}-u_h^c\|_0\lesssim \rho^{\frac{1}{2}}R_m,
\end{equation}
where $R_m := \|f\|_0 + \|g_D\|_{-\frac{1}{2},\rho_h,\Gamma_D}$. 
\end{theorem}
\begin{proof}
Taking $\bm q_h=\bm q_h^c$ in the first equation in
\eqref{eq:symm-DG}, we see that 
\begin{equation} \label{equ:WG-conforming} 
\left\{
\begin{aligned}
(c\bm p_h^{\eta}, \bm q_h^c) -(u_h^{\eta}, {\rm
div}\bm q_h^c) &= -\langle g_D, \bm q_h^c\cdot
\bm n\rangle_{\Gamma_D} \qquad\qquad \forall
\bm q_h^c\in \bm Q_{h,0}^c,\\
-({\rm div}_h\bm p_h^{\eta}, v_h)+\langle [\bm
p_h^{\eta}], \{v_h\}\rangle-\langle \check p_h^{\eta},
  [v_h]_e\rangle &= -(f, v_h)+\langle g_N,
  v_h\rangle_{\Gamma_N} ~\quad \forall v_h\in V_h.
\end{aligned}
\right.
\end{equation}
Let 
$$ 
\bm{\delta}_h^p = \bm{p}_h^\eta - \bm{p}_h^c, \quad 
\delta_h^u = u_h^\eta - u_h^c. 
$$ 
Subtracting \eqref{mixedmethod} from \eqref{equ:WG-conforming}, we
have 
$$
\left\{
\begin{aligned}
(c \bm{\delta}_h^p, \bm{q}^c_h) - (\delta_h^u,  {\rm div} \bm{q}^c_h)
 &= 0 \qquad\qquad\qquad\qquad\qquad\qquad\qquad\qquad\quad
 ~\forall \bm{q}^c_h\in \bm Q^c_{h,0},\\
({\rm div}_h\bm{\delta}_h^p, v_h) &= -\langle [\bm p_h^{\eta}],
\{v_h\}\rangle+\langle \check p_h^{\eta},
[v_h]_e\rangle+\langle g_N, v_h\rangle_{\Gamma_N}
\quad \forall v_h \in V_h. 
\end{aligned}
\right.
$$
By the assumption $\{V_h\}\subset \check V_h$ and noting that $\bm
p_h^{\eta}$ satisfies \eqref{eq:XG4}, we have 
$$
\left\{
\begin{aligned}
(c \bm{\delta}_h^p, \bm{q}^c_h) - (\delta_h^u,  {\rm div}
\bm{q}^c_h) &= 0 \qquad \qquad\qquad\qquad\qquad\qquad ~\forall
\bm{q}^c_h\in \bm Q^c_{h,0},\\
({\rm div}_h \bm{\delta}_h^p, v_h) &=
-\langle\rho h_e \check u_h^{\eta}, \{v_h\}\rangle + \langle \check p_h^{\eta}, [v_h]_e\rangle \quad
  \forall
v_h \in V_h. 
\end{aligned}
\right.
$$
Further, for any  $\bm{p}^{I}_h\in \bm Q^c_{h,g}$, 
\begin{equation}\label{mixederror1}
\left\{
\begin{aligned}
(c (\bm{p}^{I}_h-\bm{p}^{c}_h), \bm{q}^c_h) - (\delta_h^u,  {\rm div}
\bm{q}^c_h) &=
(c(\bm{p}^{I}_h-\bm{p}^{\eta}_h),\bm{q}^c_h)
\qquad\qquad\qquad \qquad\qquad\qquad\qquad\quad \forall \bm{q}^c_h\in
\bm Q^c_{h,0},\\
({\rm div} (\bm{p}^{I}_h-\bm{p}^{c}_h),v_h) &= -\langle\rho h_e \check
u_h^{\eta}, \{v_h\}\rangle+\langle \check p_h^{\eta},
[v_h]_e\rangle +({\rm
div}(\bm{p}^{I}_h-\bm{p}^{\eta}_h),v_h) \quad \forall v_h
\in V_h. 
\end{aligned}
\right.
\end{equation}
By the well-posedness of the mixed methods
\eqref{equ:well-posedness-mixed}, trace inequality, inverse inequality
and Cauchy-Schwarz inequality, we have
\begin{equation} \label{eq:div-error1} 
\begin{aligned}
& \|\bm{p}^{I}_h-\bm{p}^{c}_h\|_{\bm H({\rm div}) }+\|\delta_h^u\|_0
\\
\leq &~ C_m\Big( \sup_{\bm{q}^c_h\in \bm{Q}^c_{h,0} \setminus \{\bm 0\}}
\frac{(c(\bm{p}^{I}_h - \bm{p}^{\eta}_h),
  \bm{q}^c_h)}{\|\bm{q}^c_{h}\|_{\bm H({\rm div})}} 
\\ 
&\quad~~ +\sup_{v_h\in V_h\setminus \{ 0\}} \frac{-\langle\rho h_e
\check u_h^{\eta}, \{v_h\}\rangle+\langle \check
p_h^{\eta}, [v_h]_e\rangle+ ({\rm
  div}(\bm{p}^{I}_h-\bm{p}^{\eta}_h),v_h)}{\|v_h\|_0} \Big)\\
\lesssim &~ 
\|\bm{p}^{I}_h-\bm{p}^{\eta}_h\|_{0,c} +\|{\rm
div}_h(\bm{p}^{I}_h-\bm{p}^{\eta}_h)\|_0
+\rho^{\frac{1}{2}}(\|\check u_h^{\eta}\|_{0,\rho_h}+\|\check p_h^{\eta}\|_{0,\rho_h^{-1}}).
\end{aligned}
\end{equation}
Hence, by \eqref{eq:div-error1} and \eqref{eq:div-stability}, we have
$$ 
\begin{aligned}
& \quad~ \|\bm \delta_h^{p}\|_{0,c} +\|{\rm div}_h \bm \delta_h^p\|_0 +
\|\delta_h^u\|_0\\
& \lesssim \rho^{\frac{1}{2}}(\|\check u_h^{\eta}\|_{0,\rho_h}+\|\check p_h^{\eta}\|_{0,\rho_h^{-1}})+ \inf_{\bm p_h^I \in \bm Q_{h,g}^c}
\left( \|\bm{p}^{I}_h-\bm{p}^{\eta}_h\|_{0,c} + \|{\rm
div}_h(\bm{p}^{I}_h-\bm{p}^{\eta}_h)\|_{0,c} \right) \\
& \lesssim \rho^{\frac{1}{2}}(\|\check u_h^{\eta}\|_{0,\rho_h}+\|\check p_h^{\eta}\|_{0,\rho_h^{-1}}) + (\sum_{e\in \mathcal{E}_h} h_e^{-1}\|[\bm
p_h^\eta]\|^2_{0,e})^{\frac{1}{2}}\\
&\lesssim \rho^{\frac{1}{2}}\big(\|\check u_h^{\eta}\|_{0,\rho_h}+\|\check p_h^{\eta}\|_{0,\rho_h^{-1}}+\|\bm p_h^{\eta}\|_{{\rm div},\rho_h}\big) 
\lesssim \rho^{\frac12}R_m.
\end{aligned}
$$
This completes the proof. 
\end{proof}

\section{Unified Extended Galerkin Analysis of Existing Methods}
In this section, we exploit the relationship between the
formulation \eqref{eq:XG} and several existing numerical methods, which leads to
the well-posedness and error estimates of the existing numerical
methods.  We consider three different variants of the 4-field system
\eqref{eq:XG} by eliminating either $\check p_h$ or $\check u_h$, or
both.  


\subsection{Eliminating $\check{p}_h$} \label{subsec:HDG}
By \eqref{eq:XG3}, we have the explicit expression of $\check{p}_h$ as 
\begin{equation} \label{eq:HDG-hatp}
\check{p}_h =
\left\{
\begin{array}{ll}
\tau \check{\mathcal Q}_h^p [u_h]_e &\hbox{on}~~ \mathcal
E_h^i,\\
\tau \check{\mathcal Q}_h^p(u_h-g_D) &\hbox{on}~~ \Gamma_D, \\
0 &\hbox{on}~~\Gamma_N.
\end{array}
\right.
\end{equation}
Then formulation \eqref{eq:XG} \eqref{eq:XG} is reduced to 
\begin{equation} \label{eq:hdg-compact} 
\left\{
\begin{aligned}
a_{\rm H}(\bm p_h, \bm q_h) + b_{\rm H}(\bm q_h, \tilde{u}_h) &=
-\langle g_D,\boldsymbol {q}_h\cdot\boldsymbol n\rangle_{\Gamma_D}
\qquad\qquad\qquad\qquad\qquad\qquad\quad~ \forall \bm q_h \in \bm
Q_h, \\
b_{\rm H}(\bm p_h, \tilde{v}_h) - c_{\rm H}(\tilde{u}_h, \tilde{v}_h)
&= -(f, v_h)+\langle g_N, v_h+\check v_h\rangle_{\Gamma_N}-\langle
\tau \check{\mathcal Q}_h^p g_D,v_h\rangle_{\Gamma_D}\quad \forall
\tilde{v}_h \in \tilde V_h, 
\end{aligned}
\right.
\end{equation} 
where 
$$ 
\begin{aligned}
a_{\rm H}(\bm p_h, \bm q_h) &= (c\bm{p}_h, \bm{q}_h), \\
b_{\rm H}(\bm q_h, \tilde{u}_h) &= -(u_h,
{\rm div}_h\bm{q}_h) + \langle\check u_h+\{u_h\}, [\bm q_h]\rangle, \\
c_{\rm H}(\tilde{u}_h, \tilde{v}_h) &= \langle \eta^{-1}\check u_h, \check v_h \rangle + \langle 
\tau \check{\mathcal Q}_h^p[u_h]_e, \check{\mathcal Q}_h^p[v_h ]_e\rangle.
\end{aligned}
$$ 
Now let us transform $\hat u_h := \check{\mathcal Q}_h^u\{u_h\}+\check
u_h$, then we can rewrite the above formulation as: Find $(\boldsymbol
p_h, u_h, \hat u_h)\in \boldsymbol Q_h\times V_h\times \check V_h$
such that 
\begin{equation} \label{eq:hdg} 
\left\{
\begin{aligned}
a_{\rm H}(\bm p_h, \bm q_h) + b_{\rm H}(\bm q_h; {u}_h, \hat u_h) &=
-\langle g_D,\boldsymbol {q}_h\cdot\boldsymbol n\rangle_{\Gamma_D}
~~\qquad\qquad\qquad\qquad\qquad\qquad \qquad \qquad \quad \forall \bm
q_h \in \bm Q_h, \\
b_{\rm H}(\bm p_h; v_h, \hat{v}_h) - c_{\rm H}(u_h, \hat{u}_h;
    v_h,\hat{v}_h) &= -(f, v_h)+\langle g_N, v_h+\hat
v_h-\check{\mathcal Q}_h^u\{v_h\}\rangle_{\Gamma_N}-\langle \tau
\check{\mathcal Q}_h^pg_D,v_h\rangle_{\Gamma_D}  ~~\forall (v_h,
    \hat{v}_h) \in
\tilde V_h, 
\end{aligned}
\right.
\end{equation} 
where 
$$ 
\begin{aligned}
a_{\rm H}(\bm p_h, \bm q_h) &= (c\bm{p}_h, \bm{q}_h), \\
b_{\rm H}(\bm q_h; {u}_h, \hat u_h) &= -(u_h,
{\rm div}_h\bm{q}_h) + \langle\hat u_h-\check{\mathcal Q}_h^u\{u_h\}+\{u_h\}, [\bm q_h]\rangle, \\
 c_{\rm H}(u_h, \hat{u}_h; v_h,\hat{v}_h) &= \langle \eta^{-1}(\hat u_h-\check{\mathcal Q}_h^u\{u_h\}), \hat v_h-\check{\mathcal Q}_h^u\{v_h\} \rangle + \langle \tau
\check{\mathcal Q}_h^p[u_h]_e, \check{\mathcal Q}_h^p[ v_h]_e \rangle.
\end{aligned}
$$ 
The resulting three-field formulation \eqref{eq:hdg-compact} is a
generalization of the stabilized hybrid mixed method
\cite{hong2017unified}, or some special cases of the HDG method
\cite{cockburn2004characterization, cockburn2009unified,
cockburn2010projection, lehrenfeld2010hybrid, oikawa2015hybridized}.

\paragraph{Some special cases:}
More precisely, under the conditions that  $\eta=\frac{1}{4}\tau^{-1}$
and $\check{V}_h = \check{Q}_h$, \eqref{eq:hdg-compact} is shown to
be the standard HDG method \cite{cockburn2004characterization,
  cockburn2009unified, cockburn2010projection}, if  
\begin{equation} \label{eq:XG2standardHDG}
\bm Q_h\cdot \bm n_e |_{\mathcal E_h}\subset \check V_h \qquad
\text{and} \qquad V_h|_{\mathcal E_h}\subset \check V_h.
\end{equation}
Under the condition \eqref{eq:XG2standardHDG} and $\eta =\frac{1}{4} \tau^{-1}$, using the identity
\eqref{eq:dg-identity_2}, a hybridizable formulation of \eqref{eq:hdg} is obtained: Find
$(\bm p_h, u_h, \hat u_h)\in \bm Q_h\times V_h\times \check V_h$ such
that for any $(\bm q_h, v_h, \hat{v}_h) \in \bm Q_h\times V_h\times
\check V_h$
\begin{equation} \label{eq:hdg-H} 
\left\{
\begin{aligned}
(c\bm p_h, \bm q_h) -(u_h, {\rm div}_h\bm{q}_h) + \langle\hat u_h, \bm
q_h\cdot \bm n\rangle_{\partial\mathcal T_h}&=  -\langle g_D,\bm
{q}_h\cdot\bm n\rangle_{\Gamma_D},\\
-({\rm div}_h\bm{p}_h, v_h)+\langle 2\tau (\hat u_h-\check{\mathcal
    Q}_h^u u_h),\check{\mathcal Q}_h^u v_h \rangle_{\partial \mathcal
  T_h}  &=-(f, v_h)+\langle g_N, v_h-\check{\mathcal
    Q}_h^u\{v_h\}\rangle_{\Gamma_N}-\langle \tau \check{\mathcal
      Q}_h^pg_D,v_h\rangle_{\Gamma_D} ,\\
\langle\bm p_h\cdot \bm n,\hat v_h\rangle_{\partial\mathcal T_h} -
\langle 2\tau (\hat u_h-\check{\mathcal Q}_h^u u_h), \hat
v_h\rangle_{\partial \mathcal T_h}&=\langle g_N, \hat
v_h\rangle_{\Gamma_N}.
\end{aligned}
\right.
\end{equation} 
The above formulation shows that $\bm p_h$ and $u_h$ can be
represented by $\hat u_h$ locally from the first and the second
equations. As a result, a globally coupled equation solely for $\hat
u_h$ on $\mathcal E_h$ can be obtained.

Moreover, \eqref{eq:hdg} reduces to the HDG with reduced
stabilization method \cite{lehrenfeld2010hybrid, oikawa2015hybridized}
if 
\begin{equation} \label{eq:XG2reducedHDG}
\bm Q_h\cdot \bm n_e |_{\mathcal E_h}\subset \check V_h.
\end{equation}
Specific choices of the discrete space and the corresponding numerical
methods are summarized in Table \ref{tab:XG2HDG}.  We refer to
\cite{hong2017unified} for discussion from the HDG to the hybrid mixed
methods \cite{arnold1985mixed, brezzi1991mixed,
cockburn2004characterization} and the mixed methods
\cite{raviart1977mixed, nedelec1980mixed, brezzi1985two,
brezzi1987mixed, brezzi1991mixed, boffi2013mixed}.


\begin{remark}
We should note that the uniform inf-sup condition for the HDG method when $\eta =\frac{1}{4}\tau^{-1}=\mathcal O(1)$, $\bm Q_h=\bm Q_h^{k}$, $V_h=V_h^{k}$, $\hat{V}_h=\hat{V}_h^{k}$ is not proved in Section \ref{sec:well-posedness}.
\end{remark}


\paragraph{Minimal stabilized divergence-based method.} In light of
Theorem \ref{thm:div-infsup}, the divergence-based inf-sup condition
holds for {\it any $\check Q_h$}.  Hence, when choosing $\check{Q}_h =
\{0\}$, the formulation \eqref{eq:hdg-compact} reduces to a stabilized
divergence-based method with minimal stabilization, which reads: Find
$(\bm p_h, u_h, \check u_h)\in \bm Q_h\times V_h\times\check V_h$,
such that for any $(\bm q_h, v_h, \check v_h)\in \bm Q_h\times
V_h\times\check V_h$
\begin{equation} \label{eq:newhdg} 
\left\{
\begin{aligned}
(c\bm{p}_h, \bm{q}_h) - (u_h, {\rm div}_h\bm{q}_h) + \langle\check
u_h + \{u_h\}, [\bm q_h]\rangle&= -\langle g_D,\bm
{q}_h\cdot\bm n\rangle_{\Gamma_D}, \\
-( {\rm div}_h\bm{p}_h,v_h) +
\langle[\bm p_h],\check v_h + \{v_h\} \rangle 
-\langle \eta^{-1}\check u_h, \check v_h \rangle
&= -(f, v_h)+\langle g_N, \check v_h+v_h\rangle_{\Gamma_N}.
\end{aligned}
\right.
\end{equation} 
Consequently, the scheme \eqref{eq:newhdg} is stable provided that
$\bm{Q}_h$, $V_h$ and $\check V_h$ satisfy the conditions in Theorem
\ref{thm:div-infsup}.

Further, by assuming $\bm{Q}_h \cdot \bm{n}_e|_{\mathcal E_h} \subset
\check V_h$ and eliminating $\check u_h$ (see \eqref{eq:WG-hatu}
below), we obtain the mixed DG method \cite{hong2017unified}: Find
$(\bm p_h, u_h)\in \bm Q_h\times V_h$ such that for any $(\bm q_h,
v_h) \in \bm Q_h\times V_h$
\begin{equation} \label{eq:mixedDG}
\left\{
\begin{aligned}
(c \bm p_h, \bm q_h) + \langle \eta[\bm p_h], [\bm
q_h]\rangle  
+ (\nabla_h u_h,  \bm q_h) 
-\langle\llbracket u_h \rrbracket,
\{\vect{q_h}\}\rangle& =-\langle g_D, \bm q_h\cdot
\bm n\rangle_{\Gamma_D}+\langle\eta g_N, \bm q_h\cdot
\bm n\rangle_{\Gamma_N}, \\
(\bm {p_h},\nabla_h v_h)
-\langle\{\vect{p_h}\}, \llbracket v_h \rrbracket\rangle 
&=-(f,v_h)+\langle g_N, v_h\rangle_{\Gamma_N}.
\end{aligned}
\right.
\end{equation} 
This implies that the mixed DG method proposed in
\cite{hong2017unified} can be interpreted as the minimal stabilized
divergence-based method.  

\paragraph{Mixed method.} 
Finally, we remark that, if we take $\tau \to 0$ and choose $\bm
Q_h\times V_h\times \check V_h=\bm Q_h^{k+1}\times V_h^k\times \check
V_h^{k+1}$ or $\bm Q_h\times V_h\times \check V_h=\bm Q_h^{k,RT}\times
V_h^k\times \check V_h^{k}$, the \eqref{eq:hdg-compact} implies the
mixed method by eliminating $\hat{u}_h$.

\subsection{Eliminating $\check{u}_h$} \label{subsec:WG}
By \eqref{eq:XG4}, we have the explicit expression of $\check{u}_h$ as 
\begin{equation}\label{eq:WG-hatu}
\check u_h = 
\left\{
\begin{array}{ll}
\eta\check{\mathcal Q}_h^u [\bm p_h]&  \hbox{on}~ \mathcal
E_h^i, \\
0 & \hbox{on}~\Gamma_D, \\
\eta\check{\mathcal Q}_h^u (\bm p_h\cdot\bm n-g_N) &
\hbox{on}~ \Gamma_N .
\end{array}
\right.
\end{equation}
Then formulation \eqref{eq:XG} can be recast as  
\begin{equation} \label{eq:wg-compact} 
\left\{
\begin{aligned}
a_w(\tilde{\bm p}_h, \tilde{\bm q}_h) + b_w(\tilde{\bm q}_h, u_h) &=
-\langle g_D,\bm {q}_h\cdot\bm n+\check q_h\rangle_{\Gamma_D} +\langle
\eta \check{\mathcal Q}_h^ug_N,\bm {q}_h\cdot\bm n\rangle_{\Gamma_N}
\quad \forall \tilde{\bm q}_h \in \tilde{\bm Q}_h, \\
b_w(\tilde{\bm p}_h, v_h) &= -(f, v_h)+\langle
g_N,v_h\rangle_{\Gamma_N} \qquad\qquad\qquad\qquad\qquad\forall v_h
\in V_h, 
\end{aligned}
\right.
\end{equation} 
where 
$$ 
\begin{aligned}
a_w(\tilde{\bm p}_h, \tilde{\bm q}_h) &= (c\bm{p}_h, \bm{q}_h)+
\langle\eta\check{\mathcal Q}_h^u[\bm{p}_h], \check{\mathcal
Q}_h^u[\bm{q}_h] \rangle  + \langle \tau^{-1}\check
p_h, \check q_h\rangle, \\
b_w(\tilde{\bm q}_h, u_h) &= (\nabla_h u_h, \bm{q}_h) -
\langle [u_h]_e, \check q_h + \{\vect{q}_h\}_e\rangle.
\end{aligned}
$$ 
Now let us transform $\hat p_h :=\check{\mathcal Q}_h^p \{\bm p_h\}_e
+ \check p_h$, then we can rewrite the above formulation as: Find
$(\bm p_h, \hat p_h, u_h)\in \bm Q_h\times \check Q_h\times V_h$ such
that 
\begin{equation} \label{eq:wg} 
\left\{
\begin{aligned}
a_{w}(\bm p_h,\hat{p}_h;\bm q_h, \hat{q}_h) + b_{w}(\bm q_h, \hat q_h;
    {u}_h) &=  -\langle g_D,\bm {q}_h\cdot\bm n+\hat
q_h-\check{\mathcal Q}_h^p \{\bm q_h\}_e\rangle_{\Gamma_D}+\langle \eta
\check{\mathcal Q}_h^ug_N,\bm {q}_h\cdot\bm n\rangle_{\Gamma_N} ~~ 
\forall (\bm q_h, \hat q_h) \in \tilde{\bm Q}_h, \\
b_{ w}(\bm p_h, \hat{p}_h; v_h)&= -(f, v_h)+\langle g_N,
  v_h\rangle_{\Gamma_N}
\qquad\qquad\qquad\qquad\qquad\qquad\qquad ~~ \forall v_h \in V_h, 
\end{aligned}
\right.
\end{equation} 
where 
$$ 
\begin{aligned}
a_{w}(\bm p_h,\hat{p}_h;\bm q_h, \hat{q}_h) &= (c\bm{p}_h, \bm{q}_h)+
\langle \tau^{-1}(\hat p_h-\check{\mathcal Q}_h^p \{\bm p_h\}_e), \hat
q_h-\check{\mathcal Q}_h^p \{\bm
q_h\}_e\rangle + \langle\eta \check{\mathcal
  Q}_h^u[\bm{p}_h], \check{\mathcal Q}_h^u[\bm{q}_h] \rangle, \\
b_{w}(\bm q_h, \hat q_h; {u}_h) &= (\nabla_h u_h, \bm{q}_h) - \langle
[u_h]_e, \hat q_h-\check{\mathcal Q}_h^p \{\bm q_h\}_e+
\{\bm q_h\}_e \rangle.
\end{aligned}
$$ 
The resulting three-field formulation \eqref{eq:wg-compact} is a
generalization of the stabilized hybrid primal method
\cite{hong2017unified}, or some special cases of the WG-MFEM method \cite{wang2014weak}. 

\paragraph{Some special cases:}Again, 
under the conditions that $\tau=\frac{1}{4}\eta^{-1}$ and $\check{Q}_h =
\check{V}_h$, \eqref{eq:wg} is the WG-MFEM method \cite{wang2014weak}, if  
\begin{equation} \label{eq:XG2WG}
\bm Q_h\cdot \bm n_e|_{\mathcal E_h}\subset \check Q_h.
\end{equation}
That is, we have the following formulation: Find $(\bm p_h, \hat p_h,
u_h)\in \bm Q_h\times \check Q_h\times V_h$ such that for any $(\bm
q_h, \hat q_h, v_h)\in \bm Q_h\times \check Q_h\times V_h$
\begin{equation} \label{eq:wg-H} 
\left\{
\begin{aligned}
(c\bm{p}_h, \bm{q}_h)- \langle 2\eta(\hat p_h-\bm p_h\cdot\bm n), \bm
q_h\cdot\bm n\rangle_{\partial\mathcal T_h}+ (\nabla_h u_h,
\bm{q}_h)&= -\langle g_D,\bm {q}_h\cdot\bm n-\check{\mathcal
Q}_h^p \{\bm q_h\}_e\rangle_{\Gamma_D}+\langle\eta
\check{\mathcal Q}_h^ug_N,\bm {q}_h\cdot\bm n\rangle_{\Gamma_N}, \\
(\bm{p}_h,\nabla_h v_h) - \langle \hat p_h,
v_h\rangle_{\partial\mathcal T_h}&= -(f, v_h)+\langle g_N,
v_h\rangle_{\Gamma_N},\\
 \langle2\eta (\hat p_h-\bm p_h\cdot \bm n), \hat q_h\rangle_{\partial
   \mathcal T_h}  - \langle u_h, \hat q_h\rangle_{\partial \mathcal
     T_h} &=  -\langle g_D,\hat q_h\rangle_{\Gamma_D}. \\
\end{aligned}
\right.
\end{equation} 
Several possible discrete spaces for \eqref{eq:wg-H} are 
$$
V_h=V_h^{k+1},\bm Q_h=\bm Q_h^{k},\check Q_h=\check Q_h^{k}, \quad
\text{or} \quad  V_h=V_h^k,\bm Q_h=\bm Q_h^{k,RT},\check Q_h=\check
Q_h^k.
$$
We refer to \cite{hong2017unified} for discussion from the WG to the
hybrid primal methods \cite{raviart1975hybrid, wolf1975alternate,
raviart1977primal} and the primal methods \cite{argyris1954energy,
turner1956stiffness, mikhlin1964variational, cea1964approximation,
pian1964derivation, Feng1965finite, crouzeix1973conforming}.

\paragraph{Minimal stabilized gradient-based method.} In light of
Theorem \ref{thm:grad-infsup}, the gradient-based inf-sup condition
holds for {\it any $\check V_h$}. Hence, we relax the condition in WG
by choosing $\check V_h=\{0\}$ in \eqref{eq:wg-compact} to obtain a
stabilized gradient-based method with minimal stabilization, which
reads: Find $(\bm p_h, \check p_h, u_h)\in \bm Q_h\times \check
Q_h\times V_h$, such that for any $(\bm q_h, \check q_h, v_h)\in \bm
Q_h\times \check Q_h\times V_h$
\begin{equation} \label{equ:newwg} 
\left\{
\begin{aligned}
(c\bm{p}_h, \bm{q}_h) + \langle\tau^{-1}\check p_h, \check q_h
\rangle + (\nabla_h u_h, \bm{q}_h) - \langle [u_h]_e,
  \check{q}_h + \{\bm q_h\}_e\rangle &=  -\langle
  g_D,\bm {q}_h\cdot\bm n+\check q_h\rangle_{\Gamma_D},\\
(\bm{p}_h,\nabla_h v_h) - \langle
 \check{p}_h + \{\bm{p}_h\}_e,[v_h]_e\rangle
&= -(f, v_h)+\langle g_N,v_h\rangle_{\Gamma_N}.
\end{aligned}
\right.
\end{equation} 
Consequently, the scheme \eqref{equ:newwg} is also stable provided
that $\bm{Q}_h$, $\check Q_h$ and $V_h$ satisfy the conditions in
Theorem \ref{thm:grad-infsup}. Further, by the elimination of $\check
p_h$ using \eqref{eq:HDG-hatp}, we obtain an LDG method \cite{cockburn1998local} in mixed
form: Find $(\bm p_h, u_h)\in \bm Q_h\times V_h$ such that
for any $(\bm q_h, v_h) \in \bm Q_h\times V_h$
\begin{equation} \label{eq:LDG-0}
\left\{
\begin{aligned}
(c \bm p_h, \bm q_h) + (\nabla_h u_h,  \bm q_h)-\langle\llbracket u_h
\rrbracket, \{\vect{q_h}\}\rangle& =-\langle g_D, \bm
q_h\cdot \bm n\rangle_{\Gamma_D}, \\
(\bm {p_h},\nabla_h v_h)
-\langle\{\vect{p_h}\}, \llbracket v_h \rrbracket\rangle 
-\langle\tau\check{\mathcal Q}_h^p[u_h]_e, \check{\mathcal
  Q}_h^p[v_h]_e\rangle &=-(f,v_h)+\langle g_N, v_h\rangle_{\Gamma_N}+\langle \tau  \check{\mathcal Q}_h^pg_D, v_h\rangle_{\Gamma_D}.
\end{aligned}
\right.
\end{equation} 

\begin{table}[!htbp]
\centering
\begin{tabular}{cccc|c|c}
\hline
 $\bm{Q}_h$ & $\check{Q}_h$ & $V_h$ & $\check{V}_h$& reference &
 inf-sup condition \\ 
\hline
 $\bm Q_h^k$ & $\check{Q}_h^{k+1}$  & $V_h^{k+1}$ &$\check{V}_h^{k+1}$ & HDG in \cite{lehrenfeld2010hybrid} & gradient-based\\
\hline
$\bm Q_h^{k+1}$ & $\check{Q}_h^{k+1}$ & $V_h^{k}$ & $\check{V}_h^{k+1}$&HDG in \cite{ cockburn2009unified} & divergence-based \\
\hline
$\bm Q_h^{k,{\rm RT}}$  & $\check{Q}_h^{k}$& $V_h^{k}$ & $\check{V}_h^{k}$&HDG in \cite{ cockburn2009unified}  & divergence-based \\
\hline
 $\bm Q_h^k$  & $\check{Q}_h^{k}$ & $V_h^{k+1}$ &
$\check{V}_h^k$&HDG with reduced stabilization in \cite{lehrenfeld2010hybrid, oikawa2015hybridized} & gradient-based\\
\hline
$\bm Q_h^{k}$ & $\check{Q}_h^{k}$ & $V_h^{k}$ & $\check{V}_h^{k}$&HDG in \cite{cockburn2008superconvergent,cockburn2010projection}
& not proved \\
\hline 
 $\bm Q_h^{k+1}$  & $\{0\}$ & $V_h^{k}$ &
$\check{V}_h^k$&Mixed DG in \cite{hong2017unified} & divergence-based\\
\hline 
$\bm Q_h^{k,{\rm RT}}$  & $\check{Q}_h^{k}$& $V_h^{k}$ & $\check{V}_h^{k+1}$&WG in \cite{wang2013weak} & divergence-based \\
\hline
 $\bm Q_h^k$  & $\check{Q}_h^{k}$ & $V_h^{k+1}$ &
$\check{V}_h^k$&WG-MFEM in \cite{wang2014weak} & gradient-based \\
\hline
 $\bm Q_h^k$  & $\check{Q}_h^{k}$ & $V_h^{k+1}$ &
$\{0\}$&LDG in \cite{cockburn1998local} & gradient-based \\
\hline
\end{tabular}
\caption{From \eqref{eq:XG} to existing methods 
} 
\label{tab:XG2HDG}
\end{table}


\paragraph{Primal method.}
We remark that, if we take $\eta \to 0$ and choose $\bm Q_h\times
\check Q_h\times V_h=\bm Q_h^{0}\times \check Q_h^{0}\times V_h^1$,
the WG method \eqref{eq:wg} is equivalent to the nonconforming finite
element method discretized by Crouzeix-Raviart element. However, when
choosing choose $\bm Q_h\times \check Q_h\times V_h=\bm Q_h^{1}\times
\check Q_h^{1}\times V_h^2$ and taking $\eta\to 0$ , the WG method
\eqref{eq:wg} is getting unstable. In this case, the stabilization is
needed for the hybrid primal method which induces to the WG method.

\subsection{Eliminating both $\check{p}_h$ and $\check{u}_h$} \label{subsec:DG}
Plugging in \eqref{eq:HDG-hatp} and \eqref{eq:WG-hatu} into
\eqref{eq:XG3} and \eqref{eq:XG4}, respectively,  we obtain a DG
method: Find $(\bm p_h, u_h)\in \bm Q_h\times V_h$ such that for any
$(\bm q_h, v_h) \in \bm Q_h\times V_h$
\begin{equation} \label{eq:newDG}
\left\{
\begin{aligned}
(c \bm p_h, \bm q_h) + (\nabla_h u_h,  \bm q_h) + \langle
\eta\check{\mathcal Q}_h^u[\bm p_h], \check{\mathcal Q}_h^u[\bm
q_h]\rangle-\langle\llbracket u_h \rrbracket,
\{\vect{q_h}\}\rangle& =-\langle g_D, \bm q_h\cdot
\bm n\rangle_{\Gamma_D}+\langle \eta \check {\mathcal  Q}_h^u g_N, \bm q_h\cdot
\bm n\rangle_{\Gamma_N}, \\
(\bm {p_h},\nabla_h v_h)
-\langle\{\vect{p_h}\}, \llbracket v_h \rrbracket\rangle 
-\langle\tau \check{\mathcal Q}_h^p[u_h]_e, \check{\mathcal
  Q}_h^p[v_h]_e\rangle &=-(f,v_h)+\langle g_N, v_h\rangle_{\Gamma_N}
 +\langle \tau  \check{\mathcal Q}_h^pg_D, v_h\rangle_{\Gamma_D}.
\end{aligned}
\right.
\end{equation} 
We note that \eqref{eq:newDG} is equivalent to the formulation
\eqref{eq:XG}.  Firstly, the solution $\bm p_h, u_h$ obtained from
\eqref{eq:XG} coincides the solution of \eqref{eq:newDG}. On the other
hand, having the solution $\bm p_h, u_h$ of \eqref{eq:newDG}, by using
\eqref{eq:HDG-hatp} and \eqref{eq:WG-hatu}, we can construct
$\check{p}_h$ and $\check{u}_h$. It is straightforward to show that
$(\bm p_h, u_h,\check{p}_h,\check{u}_h)$ is the solution of
\eqref{eq:XG}. If the choice of the spaces $\bm Q_h, \check V_h, V_h, \check Q_h$ 
satisfying $[\bm Q_h]\subset \check V_h$ and $[V_h]\subset \check Q_h$, then the 
projections $\check{\mathcal Q}_h^u$ and $\check{\mathcal Q}_h^p$ reduce to identies. 
Then in this case, \eqref{eq:newDG} reduces to the LDG method proposed in \cite{castillo2000priori}.
\begin{remark}
There are four filed: $u_h, \bm p_h, \check u_h, \check p_h$. Theoretically by eliminating any 
$m$-fields for $m\le 3$, we obtain: 
$$
C_4^1+C_4^2+C_4^3=4+6+4=14
$$
namely $14$ methods. Some of the methods should be hybridized algorithms. 
These algorithms have special interesting case under special assumption, e.g.
primal method and mixed method.
\end{remark}

\section{Conclusion}
The unified formulation, presented in this paper, is a
4-field formulation that deduces most existing finite element methods
and DG method as special cases. In particular, we deduce HDG method
and WG method from the formulation and show that they can both be recast
into a DG method derived from the unified formulation.  In addition,
we prove two types of uniform inf-sup conditions for the
formulation, which naturally lead to uniform inf-sup conditions of
HDG, WG and the DG method. 

\bibliography{UnifiedFEM}
\bibliographystyle{unsrt} 

\end{document}